\documentclass[12pt,thmsa]{article}
\usepackage{amssymb}
\usepackage{amsfonts}
\usepackage{amsmath}
\usepackage[top=3.8cm,bottom=3.8cm,textwidth=16cm,centering,a4paper]{geometry}

\newtheorem{theorem}{Theorem}[section]

\newtheorem{remark}{Remark}[section]

\newtheorem{lemma}[theorem]{Lemma}

\newtheorem{definition}[theorem]{Definition}
\newenvironment{proof}[1][Proof]{\noindent\textbf{#1.} }{\ \rule{0.5em}{0.5em}}

\begin{document}

\title{The planar Schr\"{o}dinger--Poisson system with exponential critical growth: The local well-posedness and standing waves with prescribed mass}
\date{}
\author{Juntao Sun$^{a}$\thanks{%
E-mail address: jtsun@sdut.edu.cn(J. Sun)}, Shuai Yao$^{b}$\thanks{
E-mail address: shyao2019@163.com (S. Yao)}, Jian Zhang$^{c}$\thanks{
E-mail address: slgzhangjian@163.com(J. Zhang)} \\
{\footnotesize $^{a}$\emph{School of Mathematics and Statistics, Shandong
University of Technology, Zibo 255049, PR China }}\\
{\footnotesize $^{b}$\emph{School of Mathematics and Statistics, Central
South University, Changsha 410083, PR China}}\\
{\footnotesize $^{c}$\emph{Deparment of Mathematics, Zhejiang Normal
University, Jinhua 321004, PR China }}}
\maketitle

\begin{abstract}
In this paper, we investigate a class of planar Schr\"{o}dinger-Poisson systems with critical exponential growth. We
establish conditions for the local well-posedness of the Cauchy problem in
the energy space, which seems innovative as it was not discussed at all in
any previous results. By introducing some new ideas and relaxing
some of the classical growth assumptions on the nonlinearity, we show that
such system has at least two standing waves with prescribed mass, where one
is a ground state standing waves with positive energy, and the other one is
a high-energy standing waves with positive energy. In addition, with the
help of the local well-posedness, we show that the set of ground state standing waves is
orbitally stable.
\end{abstract}



\section{Introduction}

Consider the planar Schr\"{o}dinger-Poisson system of the type
\begin{equation}
\left\{
\begin{array}{ll}
i\partial _{t}\psi +\Delta \psi +\gamma w\psi +f(\psi )=0,\quad & \forall
(t,x)\in \mathbb{R}^{1+2}, \\
-\Delta w=|\psi |^{2}, &  \\
\psi (0,x)=\psi _{0}(x), &
\end{array}%
\right.  \label{1.1}
\end{equation}%
where $\psi :\mathbb{R}^{2}\times \mathbb{R}\rightarrow \mathbb{C}$ is the
(time-dependent) wave function, the function $w$
represents an Newtonian potential for a nonlocal self-interaction of the wave
function $\psi $, the coupling
constant $\gamma \in \mathbb{\mathbb{R}}$ describe the relative strength of the potential, and the sign of the $\gamma$ determines whether
the interactions of the potential are repulsive or attractive, i.e. the interaction is attractive when
$\gamma>0$, and it is repulsive when $\gamma<0$. The
function $f$ is supposed to satisfy that $f(e^{i\theta }z)=e^{i\theta }f(z)$
for $\theta \in \mathbb{R}$ and $z\in \mathbb{C}$. Such system arises from quantum mechanics \cite{BBL,CP,L0} and in semiconductor theory \cite{L1,MRS}.

An important topic is to establish conditions for the well-posedness of Cauchy problem (\ref{1.1}). From a mathematical point of view, the second equation in the system determines $w:\mathbb{R}^{2}\rightarrow \mathbb{R}$ up to harmonic functions, it is natural to choose $w$ as the Newtonian potential of $\psi^{2}$, i.e. the convolution of $\psi^{2}$ with the Green function $\Phi (x)=-\frac{1}{2\pi }\ln |x|$ of the Laplace operator. Thus the Newtonian potential $\omega$ is given by
\begin{equation*}
w=-\frac{1}{2\pi }(\ln |x|\ast\psi^{2}).
\end{equation*}
For higher dimensional cases $(N\geq 3)$, the Green function of the Laplace operator becomes a different form $\Phi (x)=\frac{1}{N(N-2)\omega_{N}}|x|^{2-N}$, where $\omega_{N}$ denotes the volume of the unit ball in $\mathbb{R}^{N}$. As a consequence, the Schr\"{o}dinger-Poisson system in higher dimensions can be viewed as a special case of the Hartree equation, and there has been a number of works on local existence, global existence, blow up in finite time and scattering theory, see \cite{AR,C1,FH,GO,GV,GV1,KLR,TZ} and references therein. However, for the two dimensional case, there seem to be quite few results on the well-posedness of the Cauchy problem (\ref{1.1}), since the Newtonian potential $w$ diverges at the spatial infinity no matter how fast $\psi$ decays. So far, we are only aware of two papers \cite{M1,M2}. More precisely, Masaki \cite{M1} proposed a new approach to deal with such nonlocal term, which can be decomposed into a sum of the linear logarithmic potential and a good remainder. By using the perturbation method, the global well-posedness for the Cauchy problem (\ref{1.1}) with $f(\psi)=|\psi|^{p-2}\psi (p>2)$ is established in the smaller Sobolev space $\mathcal{H}$ given by
\begin{equation}
\mathcal{H}:=\left\{ \psi\in H^{1}(\mathbb{R}^{2})\text{ }|\text{\ }\int_{\mathbb{R}%
^{2}}\ln (\sqrt{1+|x|^{2}})\psi^{2}dx<\infty \right\}. \label{1.4}
\end{equation}%
For two dimensional case, we note that the Sobolev embedding guarantees that every power type nonlinearity is energy subcritical. Hence, if we are to identify an energy critical nonlinearity, then it is natural to consider an exponential type one. As far as we know, the well-posedness of the Cauchy problem (\ref{1.1}) with critical exponential growth has not been concerned in the existing literature, which is the first aim of this paper.

Another interesting topic on system (\ref{1.1}) is to study the standing waves of the form $\psi (x,t)=e^{i\lambda t}u(x),$ where $\lambda \in
\mathbb{R}$ and $u:\mathbb{R}^{2}\rightarrow \mathbb{R}.$ Then system (\ref%
{1.1}) is reduced to the system
\begin{equation}
\left\{
\begin{array}{ll}
-\Delta u+\lambda u-\gamma wu=f(u)\  & \text{in}\ \mathbb{R}^{2}, \\
-\Delta w=u^{2}\  & \text{in}\ \mathbb{R}^{2}.%
\end{array}%
\right.  \label{1.6}
\end{equation}%
With this formal inversion of the second equation in system (\ref{1.6}), we obtain the following
integro-differential equation:
\begin{equation}
-\Delta u+\lambda u-\gamma (\Phi \ast |u|^{2})u=f(u),\ \forall x\in \mathbb{R%
}^{2}.  \label{1.2}
\end{equation}%
Then at least formally, the energy functional associated with equation (\ref%
{1.2}) is
\begin{equation*}
I(u)=\frac{1}{2}\int_{\mathbb{R}^{2}}(|\nabla u|^{2}+\lambda u^{2})dx+\frac{%
\gamma }{8\pi }\int_{\mathbb{R}^{2}}\int_{\mathbb{R}^{2}}\ln
(|x-y|^{2})|u(x)|^{2}|u(y)|^{2}dxdy-\int_{\mathbb{R}^{2}}F(u)dx,
\end{equation*}%
where $F(t)=\int_{0}^{t}f(s)ds$. Obviously, if $u$ is a critical point of $I$%
, then the pair $(u,\Phi \ast |u|^{2})$ is a weak solution of system (\ref%
{1.6}). However, the energy functional $I$ is not well-defined on the
natural Sobolev space $H^{1}(\mathbb{R}^{2}),$ since the logarithm term
changes sign and is neither bounded from above nor from below. Inspired by
\cite{S}, Cingolani and Weth \cite{CW} developed a variational framework of
equation (\ref{1.2}) in the smaller Hilbert space $X$, where
\begin{equation*}
X:=\left\{ u\in H^{1}(\mathbb{R}^{2})\text{ }|\text{\ }\int_{\mathbb{R}%
^{2}}\ln (1+|x|)u^{2}dx<\infty \right\} ,
\end{equation*}%
endowed with the norm%
\begin{equation*}
\Vert u\Vert _{X}^{2}:=\int_{\mathbb{R}^{2}}(|\nabla u|^{2}+u^{2}(1+\ln
(1+|x|)))dx.
\end{equation*}

We note that there are two different ways to deal with equation (\ref{1.2})
according to the role of $\lambda $:\newline
$(i)$ the frequency $\lambda $ is a fixed and assigned parameter;\newline
$(ii)$ the frequency $\lambda $ is an unknown of the problem.

For case $(i),$ one can see that solutions of equation (\ref{1.2}) can be
obtained as critical points of the functional $I$ in $X.$ This case has
attracted much attention in the last years, under various types of
nonlinearities $f$, see, for example, \cite{ACFM,AF,CT0,CT,CW,DW} and the
references therein.

Alternatively, one can look for solutions to equation (\ref{1.2}) with the
frequency $\lambda $ unknown. In this case, the real parameter $\lambda $
appears as a Lagrange multiplier, and $L^{2}$-norms of solutions are
prescribed, i.e. $\int_{\mathbb{R}^{2}}|u|^{2}dx=c$ for given $c>0,$ which
are usually called normalized solutions. This study seems particularly
meaningful from the physical point of view, since solutions of system (\ref{1.1}) conserve their mass along time, and physicists are very interested in
the stability.

Regarding the study of normalized solutions to equation (\ref{1.2}), the first contribution
 was made by Cingolani and Jeanjean \cite{CJ}. By introducing some new ideas, they obtained several
results on nonexistence, existence and multiplicity of normalized solutions
for equation (\ref{1.2}) with the power nonlinearity $f(u)=a|u|^{p-2}u,$
depending on the assumptions on $\gamma ,a,p$ and $c.$ Very recently, Alves
et al. \cite{ABM} investigated the case of exponential critical growth on
equation (\ref{1.2}). We recall that in $\mathbb{R}^{2}$, the natural growth
restriction on the function $f$ is given by the Trudinger-Moser inequality
\cite{M,T}, and we say that a function $f$ has $\alpha _{0}$-critical
exponential growth at $+\infty $ if
\begin{equation*}
\lim_{t\rightarrow +\infty }\frac{f(t)}{e^{\alpha t^{2}}-1}=\left\{
\begin{array}{ll}
0\quad & \text{for }\alpha >\alpha _{0}, \\
+\infty \quad & \text{for }0<\alpha <\alpha _{0}.%
\end{array}%
\right.  \label{1.3}
\end{equation*}%
To make it more precise, we recall below the conditions introduced in \cite%
{ABM}.

\begin{description}
\item[$(f_{1})$] $f\in C(\mathbb{R},\mathbb{R}),$ $f(0)=0$ and has a
critical exponential growth with $\alpha _{0}=4\pi ;$

\item[$(f_{2})$] $\lim_{|t|\rightarrow 0}\frac{|f(t)|}{|t|^{\tau }}=0$ for
some $\tau >3;$

\item[$(f_{3})$] there exists a constant $\mu >6$ such that
\begin{equation*}
0<\mu F(t)\leq tf(t)\text{ for all }t\in \mathbb{R}\backslash \{0\};
\end{equation*}

\item[$(f_{4})$] there exist constants $p>4$ and $\theta >0$ such that
\begin{equation*}
F(t)\geq \theta |t|^{p}\ \text{for all}\ t\in \mathbb{R}.
\end{equation*}
\end{description}

In \cite{ABM}, under conditions $(f_{1})-(f_{4})$, they found a
mountain-pass type of normalized solution when either $0<\gamma <\gamma
_{0}\ $and $0<c<1,$ or $\gamma >0$ and $0<c<c_{0}<<1$. Moreover, when $%
-f(t)=f(-t)$ is also assumed, multiple normalized solutions with negative
energy levels were obtained by using a genus approach.

In the present paper we are likewise interested in looking for normalized
solutions to equation (\ref{1.2}) with exponential critical growth. However, distinguishing from the study in \cite{ABM}, we mainly focus on the
existence of ground state and high-energy normalized solutions by relaxing
some of the classical growth assumptions on $f$. As a result, the corresponding standing waves with prescribed mass of system (\ref{1.1}) are
obtained. In addition, the orbital stability of the set of ground state standing waves is studied as well. Specifically, for any $c>0$
given, the problem we consider is the following:
\begin{equation}
\left\{
\begin{array}{ll}
-\Delta u+\lambda u-(\Phi \ast |u|^{2})u=f(u)\  & \text{in}\ \mathbb{R}^{2},
\\
\int_{\mathbb{R}^{2}}|u|^{2}dx=c>0, &
\end{array}%
\right.  \tag{$SP_{c}$}
\end{equation}%
where $f$ satisfies conditions $(f_{1}),(f_{4})$ and

\begin{itemize}
\item[$(f_{5})$] $\lim_{t\rightarrow 0}\frac{|f(t)|}{|t|}=0;$

\item[$(f_{6})$] there exists a constant $\beta>4$ such that $\frac{%
tf(t)-2F(t)}{|t|^{\beta }}$ is decreasing on $(-\infty ,0)$ and is
increasing on $(0,+\infty).$
\end{itemize}

It is easily seen that solutions of problem $(SP_{c})$ corresponds to
critical points of the energy functional $J:X\rightarrow \mathbb{R}$ given
by
\begin{equation}
J(u)=\frac{1}{2}\int_{\mathbb{R}^{2}}|\nabla u|^{2}dx+\frac{1}{4}\int_{%
\mathbb{R}^{2}}\int_{\mathbb{R}^{2}}\ln |x-y|u^{2}(x)u^{2}(y)dxdy-\int_{%
\mathbb{R}^{2}}F(u)dx  \label{e1-5}
\end{equation}%
on the constraint
\begin{equation*}
S(c):=\left\{ u\in X\text{ }|\text{ }\int_{\mathbb{R}^{2}}|u|^{2}dx=c\right%
\} .
\end{equation*}%
It is straightforward that $J$ is a well-defined and $C^{1}$ functional on $%
S(c).$

In this work we shall make a more in-depth study of the planar Schr\"{o}dinger-Poisson system in
the exponential critical case. First of all, for the local well-posedness of the
Cauchy problem (\ref{1.1}), due to the special feature of the nonlocal term, the usual integral equation
\begin{equation*}
\psi (t)=e^{it\Delta}\psi _{0}+i\int_{0}^{t}e^{i(t-s)\Delta%
}\left(f(\psi )-\frac{\gamma }{2\pi }\psi \int_{\mathbb{R}^{2}}\ln
|x-y| |\psi (y)|^{2}dy\right)ds
\end{equation*}
is not a good choice. By adopting the ideas in \cite{M1}, we can decompose the nonlocal term into a sum of the linear logarithmic potential and a good remainder. Thus we try to consider the following integral equation
\begin{equation*}
\psi (t)=e^{it\mathcal{L}}\psi _{0}+i\int_{0}^{t}e^{i(t-s)\mathcal{L}%
}\left[f(\psi )-\frac{\gamma }{2\pi }\psi \int_{\mathbb{R}^{2}} \ln \left(\frac{|x-y|}{1+|x|}\right) |\psi (y)|^{2}dy\right]ds,
\end{equation*}
where $\mathcal{L}:=\Delta-m\ln (1+|x|)$ is the new self-adjoint operator. For the purpose of further study, we work in the space $X$, not in $\mathcal{H}$ as in (\ref{1.4}), leading to different definitions of norms and different estimates from those in \cite{M1}. In addition, comparing with the power
nonlinearity case \cite{M1}, the estimates of exponential critical nonlinearity seem to be more difficult. Secondly, it seems that the geometric properties of the energy functional $J$ have not been described in \cite{ABM}. One objective of this study is to shed some light on the behavior of $J.$ As a consequence, we shall study the existence of ground state and high-energy solutions for problem $(SP_{c}).$ Thirdly, we relax some of the
classical growth assumptions on $f$. For example, we introduce condition $(f_{5})$ instead of condition $(f_{2})$ to find the first solution of
problem $(SP_{c}),$ although it may be not a ground state. However, if the monotonicity condition $(f_{6})$ is further assumed, then such solution is a
ground state with positive energy. Moreover, as one may observe, condition $(f_{3})$ which is usually called the Ambrosetti--Rabinowitz condition was
required in \cite{ABM}. It was used in a technical but essential way in
obtaining bounded constrained Palais-Smale sequences. We shall show that,
under condition $(f_{6})$ that is weaker than condition $(f_{3})$, we manage
to extend the previous results on the existence of mountain-pass solution
for problem $(SP_{c}),$ which is the second solution of problem $(SP_{c})$
and is a high-energy solution with positive energy.

\subsection{Main results}

First of all, we establish conditions for the local well-posedness of the
Cauchy problem (\ref{1.1}) in the energy space $X.$

\begin{theorem}
\label{T5} Assume that $\Vert \nabla \psi _{0}\Vert _{L^{2}}^{2}<1$ and $f$
satisfies $(f_{1})$ and\newline
$(f_{7})$ for any $z_{1},z_{2}\in \mathbb{C}$ and $\varepsilon >0$, there
exists $C_{\varepsilon }>0$ such that
\begin{equation*}
|f(z_{1})-f(z_{2})|\leq C_{\varepsilon }|z_{1}-z_{2}|\sum_{j=1}^{2}\left(
e^{4\pi (1+\varepsilon )|z_{j}|^{2}}-1\right) ,
\end{equation*}%
and
\begin{equation*}
|f^{\prime }(z_{1})-f^{\prime }(z_{2})|\leq C_{\varepsilon
}|z_{1}-z_{2}|\sum_{j=1}^{2}\left( |z_{j}|+e^{4\pi (1+\varepsilon
)|z_{j}|^{2}}-1\right) .
\end{equation*}%
Then there exist $T_{\max }>0$ and a unique solution $\psi \in C([0,T_{\max
}],X)$ for the Cauchy problem (\ref{1.1}).
\end{theorem}

Next, we consider the following local minimization problem:%
\begin{equation}
\gamma _{c}^{\rho }:=\inf_{u\in S(c)\cap \mathcal{B}_{\rho }}J(u),
\label{e1-4}
\end{equation}%
where
\begin{equation}
\mathcal{B}_{\rho }:=\left\{ u\in X\text{ }|\text{ }\int_{\mathbb{R}%
^{2}}|\nabla u|^{2}dx\leq \rho \right\} \text{ for }\rho >0\text{ given.}
\label{e1-6}
\end{equation}%
We are now in a position to state the following result.

\begin{theorem}
\label{T1} Assume that conditions $(f_{1})$ and $(f_{5})$ hold. In addition,
we assume that $F(t)\geq 0$ for $t>0.$ Then for any $0<\rho <1,$ there
exists $0<c_{\ast }=c_{\ast }(\rho )<1$ such that for $0<c<c_{\ast },$ the
infimum $\gamma _{c}^{\rho }$ defined as (\ref{e1-4}) is achieved by $%
u_{c}\in X,$ which is a weak solution of problem $(SP_{c})$ with some $%
\lambda =\lambda _{c}\in \mathbb{\mathbb{R}}.$
\end{theorem}

\begin{definition}
\label{D1} We say that $u_{0}$ is a ground state of problem $(SP_{c})$ on $%
S(c)$ if it is a solution to problem $(SP_{c})$ having minimal energy among
all the solutions which belongs to $S(c).$ Namely,%
\begin{equation*}
(J|_{S(c)})^{\prime }(u_{0})=0\text{ and }J(u_{0})=\inf \{J(u)\text{ }|\text{%
\ }(J|_{S(c)})^{\prime }(u)=0\text{ and }u\in S(c)\}.
\end{equation*}
\end{definition}

In Theorem \ref{T1} we are not sure whether the solution $u_{c}$ is a ground
state. However, if we further assume that condition $(f_{6})$ holds, then
such solution is actually a ground state of problem $(SP_{c}).$ We have the
following result.

\begin{theorem}
\label{T2} Assume that conditions $(f_{1})$ and $(f_{5})-(f_{6})$ hold. In
addition, we assume that $F(t)\geq 0$ for $t>0.$ Let $u_{c}$ be given in
Theorem \ref{T1}. Then there exists a constant $0<\tilde{c}_{\ast }<c_{\ast
} $ such that for $0<c<\tilde{c}_{\ast },$ $u_{c}$ is a ground state to
problem $(SP_{c})$ with some $\lambda =\lambda _{c}\in \mathbb{\mathbb{R}},$ which satisfies $J(u_{c})=\gamma _{c}^{\rho }>0.$ Furthermore, there
holds%
\begin{equation*}
\gamma _{c}^{\rho }\rightarrow 0\text{ and }\int_{\mathbb{R}^{2}}|\nabla
u_{c}|^{2}dx\rightarrow 0\text{ as }c\rightarrow 0.
\end{equation*}
\end{theorem}

By Theorem \ref{T2}, we know that the set of ground states
\begin{equation*}
\mathcal{M}_{c}^{\rho }:=\left\{ u\in S(c)\cap \mathcal{B}_{\rho }\text{ }|%
\text{ }J(u)=\gamma _{c}^{\rho }\right\}
\end{equation*}%
is not empty. Then we can consider the stability of the set of ground states.

\begin{theorem}
\label{T3} Under the assumptions of Theorems \ref{T5} and \ref{T2}, the set
of ground states
\begin{equation*}
\mathcal{M}_{c}^{\rho }:=\left\{ u\in S(c)\cap \mathcal{B}_{\rho }\text{ }|%
\text{ }J(u)=\gamma _{c}^{\rho }\right\} \neq \emptyset
\end{equation*}%
is stable under the flow corresponding to (\ref{1.1}). That is, for any $%
\varepsilon >0$, there exists $\delta >0$ such that for any $\psi _{0}\in X$
satisfying $dist_{X}(\psi _{0},\mathcal{M}_{c}^{\rho })<\delta ,$ the
solution $\psi (t,\cdot )$ of (\ref{1.1}) with $\psi (0,\cdot )=\psi _{0}$
satisfies
\begin{equation*}
\sup_{t\in \lbrack 0,T)}dist_{X}(\psi (t,\cdot ),\mathcal{M}_{c}^{\rho
})<\varepsilon ,
\end{equation*}%
where $T$ is the maximal existence time for $\psi (t,\cdot ).$
\end{theorem}

Now, we turn to find the second solution of problem $(SP_{c}).$ In addition
to conditions $(f_{1}),(f_{4})$ and $(f_{6}),$ we also need condition $%
(f_{2})$ which is obviously stronger than condition $(f_{5}).$ Then we have
the following result.

\begin{theorem}
\label{T4} Assume that conditions $(f_{1})-(f_{2}),(f_{4})$ and $(f_{6})$
hold. Then there exists $0<c^{\ast }<1$ such that for $0<c<c^{\ast },$ there
exists $\theta ^{\ast }=\theta ^{\ast }(c)>0$ such that problem $(SP_{c})$
has a second pair of solutions $(\hat{u}_{c},\hat{\lambda}_{c})\in H^{1}(%
\mathbb{R}^{2})\times \mathbb{R}$ for any $\theta >\theta ^{\ast },$ which
satisfies
\begin{equation*}
J(\hat{u}_{c})>J(u_{c})=\gamma _{c}^{\rho }>0.
\end{equation*}%
In particular, $\hat{u}_{c}$ is a high-energy solution to problem $(SP_{c})$
with $\lambda =\hat{\lambda}_{c}.$
\end{theorem}

\begin{remark}
\label{R1} $(i)$ We easily find some examples of exponential critical
nonlinearities satisfying conditions $(f_{1})$, $(f_{5})$ and $(f_{7}),$ such as%
\begin{equation*}
f(t)=|t|^{p-2}te^{4\pi |t|^{2}}\text{ for all }t\in \mathbb{%
\mathbb{R}
},
\end{equation*}%
where $p>2.$ In particular, if $2<p\leq 4,$ such functions do not satisfy
condition $(f_{2});$\newline
$(ii)$ We choose a primitive function of $f$ like $F(t)=|t|^{p}e^{4\pi
|t|^{2}}$ for $p>4.$ Then we have%
\begin{equation*}
f(t)=p|t|^{p-2}te^{4\pi |t|^{2}}+8\pi |t|^{p}te^{4\pi |t|^{2}}.
\end{equation*}%
A direct calculation shows that
\begin{equation*}
\frac{tf(t)-2F(t)}{|t|^{\beta }}=(p-2+8\pi |t|^{2})|t|^{p-\beta }e^{4\pi
|t|^{2}},
\end{equation*}%
which implies that condition $(f_{6})$ holds if we take $\beta=p.$ However,
we note that we cannot find some $\mu >6$ such that
\begin{equation*}
tf(t)-\mu F(t)\geq 0\text{ for all }t\in \mathbb{R}\backslash \{0\},
\end{equation*}%
which indicates that condition $(f_{3})$ is not satisfied.
\end{remark}

\begin{remark}
\label{R2} It is necessarily mentioned that Olves et al. \cite{AJM1} studied
the existence of a mountain-pass type of normalized solution for a class of
Schr\"{o}dinger equations with exponential critical growth in $\mathbb{R}%
^{2} $. They used the standard Ambrosetti-Rabinowitz condition on the
nonlinearity $f,$ i.e.\newline
$(f_{3})^{\prime }$ there exists a constant $\mu >4$ such that
\begin{equation*}
0<\mu F(t)\leq tf(t)\text{ for all }t\in \mathbb{R}\backslash \{0\}.
\end{equation*}%
However, we observe that in the study of planar Schr\"{o}dinger--Poisson
systems, we cannot make the constant $\mu >4$ if the Ambrosetti-Rabinowitz
condition is used. The main reason is that the logarithm term changes sign.
Indeed, if we keep using the Ambrosetti-Rabinowitz condition in this
direction, then we can assume that\newline
$(f_{3})^{\prime \prime }$ for any $\kappa >1,$ there exists $\mu >4+\frac{2%
}{\kappa -1}$ such that
\begin{equation*}
0<\mu F(t)\leq tf(t)\text{ for all }t\in \mathbb{R}\backslash \{0\}.
\end{equation*}%
Obviously it is weaker than condition $(f_{3}).$
\end{remark}

The paper is organized as follows. After giving some preliminary results in Section 2, we prove Theorem \ref{T5} in Section 3 and Theorems \ref{T1}, \ref{T2} and \ref{T3} in Section 4, respectively. Finally, we give the proof of Theorem \ref{T3} in Section 5.

\section{Preliminary results}

For sake of convenience, we set%
\begin{equation*}
A(u):=\int_{\mathbb{R}^{2}}|\nabla u|^{2}dx\ \text{and}\ V(u):=\int_{\mathbb{%
R}^{2}}\int_{\mathbb{R}^{2}}\ln |x-y|u^{2}(x)u^{2}(y)dxdy.
\end{equation*}%
Then the functional $J$ defined in (\ref{e1-5}) can be reformulated as:%
\begin{equation*}
J(u)=\frac{1}{2}A(u)+\frac{1}{4}V(u)-\int_{\mathbb{R}^{2}}F(u)dx.  \label{J}
\end{equation*}

In what follows, we recall several important inequalities which will be
often used in the paper.

\textbf{(1) Hardy-Littlewood-Sobolev inequality (\cite{LL}):} Let $t,r>1$
and $0<\alpha <N$ with $1/t+(N-\alpha )/N+1/r=2$. For $\bar{f}\in L^{t}(%
\mathbb{R}^{N})$ and $\bar{h}\in L^{r}(\mathbb{R}^{N})$, there exists a
sharp constant $C(t,N,\alpha ,r)$ independent of $u$ and $v$, such that
\begin{equation}
\int_{\mathbb{R}^{2}}\int_{\mathbb{R}^{2}}\frac{\bar{f}(x)\bar{h}(y)}{%
|x-y|^{N-\alpha }}dxdy\leq C(t,N,\alpha ,r)\Vert \bar{f}\Vert _{L^{t}}\Vert
\bar{h}\Vert _{L^{r}}.  \label{2-00}
\end{equation}

\textbf{(2) Gagliardo-Nirenberg inequality (\cite{W}):} For every $N\geq 1$
and $r\in (2,2^{\ast }),$ here $2^{\ast }:=\infty $ for $N=1,2$ and $2^{\ast
}:=2N/(N-2)$ for $N\geq 3,$ there exists a sharp constant $\mathcal{S}_{r}>0$
depending on $r$ such that
\begin{equation}
\Vert u\Vert _{r}\leq \mathcal{S}_{r}^{1/r}\Vert \nabla u\Vert _{L^{2}}^{%
\frac{r-2}{r}}\Vert u\Vert _{L^{2}}^{\frac{2}{r}},  \label{GN}
\end{equation}%
where $\mathcal{S}_{r}=\frac{r}{2\Vert U\Vert _{2}^{r-2}}$ and $U$ is the
ground state solution of the following equation
\begin{equation*}
-\Delta u+\frac{2}{r-2}u=\frac{2}{r-2}|u|^{r-2}u.
\end{equation*}

As in the introduction, following \cite{CW, S}, we shall work in the Hilbert
space
\begin{equation*}
X=\left\{ u\in H^{1}(\mathbb{R}^{2})\text{ }|\text{\ }\Vert u\Vert _{\ast
}<\infty \right\} ,  \label{2.3}
\end{equation*}%
where
\begin{equation*}
\Vert u\Vert _{\ast }^{2}:=\int_{\mathbb{R}^{2}}\ln (1+|x|)u^{2}(x)dx,
\end{equation*}%
with $X$ endowed with the norm given by
\begin{equation*}
\Vert u\Vert _{X}^{2}:=\Vert u\Vert _{H^{1}}^{2}+\Vert u\Vert _{\ast }^{2}.
\end{equation*}%
Define the symmetric bilinear forms
\begin{eqnarray*}
(u,v) &\mapsto &B_{1}(u,v)=\int_{\mathbb{R}^{2}}\int_{\mathbb{R}^{2}}\ln
(1+|x-y|)u(x)v(y)dxdy, \\
(u,v) &\mapsto &B_{2}(u,v)=\int_{\mathbb{R}^{2}}\int_{\mathbb{R}^{2}}\ln
\left( 1+\frac{1}{|x-y|}\right) u(x)v(y)dxdy, \\
(u,v) &\mapsto &B_{0}(u,v)=B_{1}(u,v)-B_{2}(u,v)=\int_{\mathbb{R}^{2}}\int_{%
\mathbb{R}^{2}}\ln |x-y|u(x)v(y)dxdy,
\end{eqnarray*}%
and define the associated functional on $X$
\begin{eqnarray*}
V_{1}(u) &=&B_{1}(u^{2},u^{2})=\int_{\mathbb{R}^{2}}\int_{\mathbb{R}^{2}}\ln
(1+|x-y|)u^{2}(x)u^{2}(y)dxdy, \\
V_{2}(u) &=&B_{2}(u^{2},u^{2})=\int_{\mathbb{R}^{2}}\int_{\mathbb{R}^{2}}\ln
\left( 1+\frac{1}{|x-y|}\right) u^{2}(x)u^{2}(y)dxdy.
\end{eqnarray*}%
Obviously, one can see that%
\begin{equation*}
V(u)=V_{1}(u)-V_{2}(u).
\end{equation*}%
Note that
\begin{equation*}
\ln (1+|x-y|)\leq \ln (1+|x|+|y|)\leq \ln (1+|x|)+\ln (1+|y|)\ \text{for all}%
\ x,y\in \mathbb{R}^{2}.
\end{equation*}%
Then we have
\begin{eqnarray}
B_{1}(uv,wz) &\leq &\int_{\mathbb{R}^{2}}\int_{\mathbb{R}^{2}}(\ln
(1+|x|)+\ln (1+|y|))|u(x)v(x)||w(y)z(y)|dxdy  \notag \\
&\leq &\Vert u\Vert _{\ast }\Vert v\Vert _{\ast }\Vert w\Vert _{L^{2}}\Vert
z\Vert _{L^{2}}+\Vert u\Vert _{L^{2}}\Vert v\Vert _{L^{2}}\Vert w\Vert
_{\ast }\Vert z\Vert _{\ast }  \label{V1}
\end{eqnarray}%
for all $u,v,w,z\in L^{2}(\mathbb{R}^{2}).$ Using the fact of $0\leq \ln
(1+t)\leq t$ for $t\geq 0$, it follows from the Hardy-Littlewood-Sobolev
inequality (\ref{2-00}) that for some $\bar{K}>0,$%
\begin{equation}
|B_{2}(u,v)|\leq \int_{\mathbb{R}^{2}}\int_{\mathbb{R}^{2}}\frac{1}{|x-y|}%
u(x)v(y)dxdy\leq \bar{K}\Vert u\Vert _{L^{\frac{4}{3}}}\Vert v\Vert _{L^{%
\frac{4}{3}}}.  \label{2.1}
\end{equation}%
Hence, according to (\ref{GN}) and (\ref{2.1}), one can see that for some $%
K>0$,
\begin{equation}
|V_{2}(u)|\leq \bar{K}\Vert u\Vert _{L^{\frac{8}{3}}}^{4}\leq
KA(u)^{1/2}\Vert u\Vert _{L^{2}}^{3}\ \text{for all}\ u\in H^{1}(\mathbb{R}%
^{2}),  \label{V2}
\end{equation}%
and $V_{2}$ only takes finite values on $L^{8/3}(\mathbb{R}^{2})$.

Now we introduce some known results from \cite{CJ,CW}, which are important
in our work.

\begin{lemma}
\label{L2.1}(\cite[Lemma 2.2]{CW}) The following statements are true.\newline
$(i)$ The space $X$ is compactly embedded in $L^{r}(\mathbb{R}^{2})$ for all
$2\leq r<\infty ;$\newline
$(ii)$ The functionals $V,V_{1},V_{2}$ and $J$ are of $C^{1}$ on $X$.
Moreover, $V_{i}^{\prime }(u)[v]=4B_{i}(u^{2},uv)$ for all $u,v\in X$ and $%
i=1,2.$\newline
$(iii)$ $V_{1}$ is weakly lower semicontinuous on $H^{1}(\mathbb{R}^{2});$%
\newline
$(iv)$ $V_{2}$ is continuous (in fact, continuously differentiable) on $%
L^{8/3}(\mathbb{R}^{2}).$
\end{lemma}

\begin{lemma}
\label{L2.2}(\cite[Lemma 2.1]{CW}) Let $\{u_{n}\}$ be a sequence in $L^{2}(%
\mathbb{R}^{2})$ such that $u_{n}\rightarrow u\in L^{2}(\mathbb{R}%
^{2})\backslash \{0\}$ pointwise a.e. on $\mathbb{R}^{2}$. Moreover, let $%
\{v_{n}\}$ be a bounded sequence in $L^{2}(\mathbb{R}^{2})$ such that
\begin{equation*}
\sup_{n\in \mathbb{N}}B_{1}(u_{n}^{2},v_{n}^{2})<\infty .
\end{equation*}%
Then there exists $n_{0}\in \mathbb{N}$ and $C>0$ such that $\Vert
v_{n}\Vert _{\ast }<C$ for $n\geq n_{0}.$ Moreover, if $%
B_{1}(u_{n}^{2},v_{n}^{2})\rightarrow 0\ $and$\ \Vert v_{n}\Vert
_{2}\rightarrow 0\ $as$\ n\rightarrow \infty ,$ then
\begin{equation*}
\Vert v_{n}\Vert _{\ast }\rightarrow 0\ \text{as}\ n\rightarrow \infty .
\end{equation*}
\end{lemma}

\begin{lemma}
\label{L2.3}(\cite[Lemma 2.6]{CW}) Let $\{u_{n}\},\{v_{n}\}$ and $\{w_{n}\}$
be bounded sequence in $X$ such that $u_{n}\rightharpoonup u$ weakly in $X$.
Then for every $z\in X$, we have
\begin{equation*}
B_{1}(v_{n}w_{n},z(u_{n}-u))\rightarrow 0\ \text{as}\ n\rightarrow \infty .
\end{equation*}
\end{lemma}

\begin{lemma}
\label{L2.4}(\cite[Lemma 2.6]{CJ}) Let $\{u_{n}\}\subset S(c)$ be a sequence
such that $V_{1}(u_{n})$ is bounded. Then there exists a subsequence of $%
\{u_{n}\}$, up to translation, converging to $u$ in $L^{2}(\mathbb{R}^{2})$.
More precisely, for all $k\geq 1$, there exists $n_{k}\rightarrow \infty $
and $x_{k}\in \mathbb{R}^{2}$ such that $u_{n_{k}}(\cdot -x_{k})\rightarrow
u $ strongly in $L^{2}(\mathbb{R}^{2})$. In addition, if the sequence $%
\{u_{n}\}$ consists of radial functions, then necessarily the sequence $%
x_{k}\in \mathbb{R}^{2}$ is bounded.
\end{lemma}

We recall the well-known Moser-Trudinger inequality as follows.

\begin{lemma}
\label{L2-7}(\cite{C}) If $\alpha >0$ and $u\in H^{1}(\mathbb{R}^{2})$, then
we have
\begin{equation*}
\int_{\mathbb{R}^{2}}(e^{\alpha u^{2}}-1)dx<+\infty .
\end{equation*}%
Moreover, if $\Vert \nabla u\Vert _{2}^{2}\leq 1$, $\Vert u\Vert _{2}\leq
M<+\infty $ and $0<\alpha <4\pi $, then there exists a constant $L(M,\alpha
)>0,$ depending only on $M$ and $\alpha $, such that
\begin{equation*}
\int_{\mathbb{R}^{2}}(e^{\alpha u^{2}}-1)dx\leq L(M,\alpha ).
\end{equation*}
\end{lemma}

\begin{lemma}
\label{L2-10} Assume that conditions $(f_{1})-(f_{2})$ hold. Let $\left\{
u_{n}\right\} $ be a sequence in $S(c)$ satisfying $\limsup_{n\rightarrow
\infty }\Vert \nabla u_{n}\Vert _{2}^{2}<1-c.$ If $u_{n}\rightharpoonup u$
in $X$ and $u_{n}(x)\rightarrow u(x)$ a.e. in $\mathbb{R}^{2}$, then there
hold%
\begin{equation*}
F(u_{n})\rightarrow F(u)\quad \text{and}\quad f(u_{n})u_{n}\rightarrow
f(u)u\quad \text{in }L^{1}(\mathbb{R}^{2}).
\end{equation*}
\end{lemma}

\begin{proof}
The proof is similar to that of \cite[Corollary 3.2]{AJM1}, we omit it here.
\end{proof}

\begin{lemma}
\label{L2.6}(\cite[Lemma 3.4]{ABM}) Let $\left\{ u_{n}\right\} $ be a
sequence in $S(c)$ satisfying $\limsup_{n\rightarrow \infty }\Vert \nabla
u_{n}\Vert _{2}^{2}<1-c$ and $J(u_{n})\leq d$ for some $d\in\mathbb{R}$ and
for all $n\in\mathbb{N} .$ Then, up to a subsequence, $\left\{ u_{n}\right\}
$ is bounded in $X.$
\end{lemma}

\begin{lemma}
\label{L2.7} (The Pohozaev identity) Any weak solution $u\in X$ to the
equation
\begin{equation}
-\Delta u+\lambda u+(\ln |\cdot |\ast u^{2})u=f(u)  \label{2.2}
\end{equation}%
satisfies the Pohozaev identity
\begin{equation*}
\lambda \Vert u\Vert _{2}^{2}+\int_{\mathbb{R}^{2}}\int_{\mathbb{R}^{2}}\ln
(|x-y|)|u(x)|^{2}|u(y)|^{2}dxdy+\frac{1}{4}\Vert u\Vert _{2}^{4}-\int_{%
\mathbb{R}^{2}}2F(u)dx=0.
\end{equation*}%
In particular, it satisfies
\begin{equation*}
Q(u):=A(u)-\frac{1}{4}\Vert u\Vert _{2}^{4}+\int_{\mathbb{R}%
^{2}}(2F(u)-f(u)u)dx=0.
\end{equation*}
\end{lemma}

\begin{proof}
As the argument in \cite[Lemma 2.7]{CJ} and \cite[Lemma 2.1]{LM}, the proof
can be done by multiplying Eq. (\ref{2.2}) by $x\cdot \nabla u$ and
integrating by parts. So, we omit it here.
\end{proof}

\begin{lemma}
\label{L2.10} Assume that conditions $(f_{1})$ and $(f_{5})-(f_{6})$ hold.
Then we have
\begin{equation}
g(t,v):=t^{-2}F(tv)-F(v)+\frac{1-t^{p-2}}{p-2}[f(v)v-2F(v)]\geq 0,\ \text{for%
}\ t>0\ \text{and }v\in \mathbb{R}.  \label{2.9}
\end{equation}%
Moreover, there holds%
\begin{equation}
\frac{F(t)}{|t|^{p-1}t}\ \text{is nondecreasing on}\ (-\infty ,0)\cup
(0,+\infty ).  \label{2.10}
\end{equation}
\end{lemma}

\begin{proof}
For any $t>0$ and $v\in \mathbb{R}$, by condition $(f_{1})$, a direct
calculation shows that
\begin{eqnarray*}
g^{\prime }(t,v)|_{t} &=&-2t^{-3}F(tv)+t^{-2}f(tv)v-t^{p-3}(f(v)v-2F(v)) \\
&=&t^{p-3}|v|^{p}\left( \frac{f(tv)tv-2F(tv)}{|tv|^{p}}-\frac{f(v)v-2F(v)}{%
|v|^{p}}\right).
\end{eqnarray*}%
Using this, together with condition $(f_{6})$, yields that $g^{\prime
}(t,v)|_{t}\geq 0$ if $t\geq 1$ and $g^{\prime }(t,v)|_{t}< 0$ if $0<t< 1$.
This implies that $g(t,v)\geq g(1,v)=0$ for all $t>0$ and $v\in \mathbb{R}$.
So (\ref{2.9}) holds. Furthermore, by (\ref{2.9}) and condition $(f_{5})$,
one has
\begin{equation*}
\lim_{t\rightarrow 0}g(t,v)=\frac{1}{p-2}[f(v)v-pF(v)]\geq 0,\ \forall v\in
\mathbb{R},
\end{equation*}%
which leads to
\begin{equation*}
\left( \frac{F(t)}{|t|^{p-1}t}\right) ^{\prime }=\frac{1}{|t|^{p+1}}%
[f(t)t-pF(t)]\geq 0,\text{ }\forall t\in \mathbb{R}.
\end{equation*}%
This shows that (\ref{2.10}) holds. We complete the proof.
\end{proof}

Define a function $\Phi :\mathbb{R}\rightarrow \mathbb{R}$ given by
\begin{equation}
\Phi (t)=\frac{t^{2}}{2}A(u)-t^{-2}\int_{\mathbb{R}^{2}}F(tu)dx.  \label{J2}
\end{equation}%
Then we have the following lemma.

\begin{lemma}
\label{L2.5} Assume that conditions $(f_{1})$ and $(f_{5})-(f_{6})$ hold.
Then for any $u\in H^{1}(\mathbb{R}^{2}),$ we have
\begin{equation*}
J(u)-\Phi (t)\geq \frac{1-t^{p-2}}{p-2}Q(u)+\frac{h(t)}{2(p-2)}A(u)-\frac{K}{%
4}\Vert u\Vert _{L^{2}}^{3}A(u)^{1/2}+\frac{1-t^{p-2}}{4(p-2)}\Vert u\Vert
_{L^{2}}^{4},
\end{equation*}%
for $t\geq 0$, where
\begin{equation*}
h(t):=2t^{p-2}-(p-2)t^{2}+p-4>0\text{ for }t\geq 0.
\end{equation*}%
In particular, there holds
\begin{equation*}
J(u)\geq \frac{1}{p-2}Q(u)+\frac{p-4}{2(p-2)}A(u)-\frac{K}{4}\Vert u\Vert
_{L^{2}}^{3}A(u)^{1/2}+\frac{1}{4(p-2)}\Vert u\Vert _{L^{2}}^{4}.
\label{3.2.12}
\end{equation*}
\end{lemma}

\begin{proof}
By (\ref{V2}), (\ref{J2}) and Lemma \ref{L2.10}, for all $u\in H^{1}(\mathbb{%
R}^{2}),$ we have
\begin{eqnarray}
J(u)-\Phi (t) &=&\frac{1-t^{2}}{2}A(u)+\frac{1}{4}V(u)+\int_{\mathbb{R}%
^{2}}[t^{-2}F(tu)-F(u)]dx  \notag \\
&=&\frac{1-t^{p-2}}{p-2}Q(u)+\left[ \frac{1-t^{2}}{2}-\frac{1-t^{p-2}}{p-2}%
\right] A(u)  \notag \\
&&+\frac{1}{4}V(u)+\frac{1-t^{p-2}}{4(p-2)}\Vert u\Vert _{L^{2}}^{4}  \notag
\\
&&+\int_{\mathbb{R}^{2}}\left[ t^{-2}F(tu)-F(u)+\frac{1-t^{p-2}}{p-2}\left(
f(u)u-2F(u)\right) \right] dx  \notag \\
&\geq &\frac{1-t^{p-2}}{p-2}Q(u)+\frac{2t^{p-2}-(p-2)t^{2}+p-4}{2(p-2)}A(u)
\notag \\
&&-\frac{K}{4}\Vert u\Vert _{L^{2}}^{3}A(u)^{1/2}+\frac{1-t^{p-2}}{4(p-2)}%
\Vert u\Vert _{L^{2}}^{4}\text{ for }t\geq 0.  \label{2.2.13}
\end{eqnarray}

Set%
\begin{equation*}
h(t):=2t^{p-2}-(p-2)t^{2}+p-4\text{ for }t\geq 0.
\end{equation*}%
A direct calculation shows that $h^{\prime }(t)=2(p-2)t(t^{p-4}-1).$ Since $%
p>4,$ we have $h^{\prime }(t)\geq 0$ if $t\geq 1$ and $h^{\prime }(t)\leq 0$
if $0\leq t<1,$ which implies that $h(t)\geq h(1)=0$ for $t\geq 0$.

Letting $t\rightarrow 0$ in (\ref{2.2.13}), by condition $(f_{5})$, we
deduce that
\begin{equation*}
J(u)\geq \frac{1}{p-2}Q(u)+\frac{p-4}{2(p-2)}A(u)-\frac{K}{4}\Vert u\Vert
_{L^{2}}^{3}A(u)^{1/2}+\frac{1}{4(p-2)}\Vert u\Vert _{L^{2}}^{4}.
\end{equation*}%
We complete the proof.
\end{proof}

\begin{remark}
\label{R3} Clearly, if condition $(f_{5})$ is replaced by condition $%
(f_{2}), $ then the above lemma still holds.
\end{remark}

\section{The local well-posedness for the Cauchy problem}

We consider the local well-posedness for the Cauchy problem (\ref{1.1}). Following
the ideas in \cite{M1}, we can decompose the nonlinearity as
\begin{equation*}
\gamma \omega \psi =-\frac{\gamma }{2\pi }\Vert \psi \Vert _{2}^{2}(\ln
(1+|x|))\psi -\frac{\gamma }{2\pi }\psi \int_{\mathbb{R}^{2}}\ln \left(\frac{|x-y|}{1+|x|}\right)|\psi (y)|^{2}dy.
\end{equation*}%
Since the conservation of mass $\Vert \psi (t)\Vert _{2}=\Vert \psi
_{0}\Vert _{2}$, thus we set $m:=\frac{\gamma }{2\pi }\Vert \psi _{0}\Vert
_{2}^{2}>0$, and we obtain the following equivalent equation
\begin{equation}
\left\{
\begin{array}{ll}
i\partial _{t}\psi +(\Delta -m\ln (1+|x|))\psi -\frac{\gamma }{2\pi }\psi
\int_{\mathbb{R}^{2}}\ln \left(\frac{|x-y|}{1+|x|}\right)|\psi (y)|^{2}dy+f(\psi )=0, &
\\
\psi (0,x)=\psi _{0}(x), &
\end{array}%
\right.  \label{e2}
\end{equation}%
where the operator $\mathcal{L}:=\Delta -m\ln (1+|x|)$ is a self-adjoint
operator on $C_{0}^{\infty }(\mathbb{R}^{2})$. Since the potential $\ln (1+|x|)$
is subquadratic, for $t\in \lbrack -T,T]$, we have
\begin{equation*}
\left\Vert e^{it\mathcal{L}}\varphi \right\Vert _{L^{\infty }}\lesssim
|t|^{-1}\left\Vert \varphi \right\Vert _{L^{1}}.
\end{equation*}

\begin{definition}
The pair $(q,r)$ is referred to be as a Strichartz admissible if
\begin{equation*}
\frac{2}{q}+\frac{2}{r}=1,\quad \text{for }q,r\in \lbrack 2,\infty ],\text{
and }(q,r,2)\neq (2,\infty ,2).
\end{equation*}
\end{definition}

\begin{lemma}[Strichartz estimats]\label{L33}
For any $T>0$, the following properties hold:\newline
$(i)$ Let $\varphi\in L^{2}(\mathbb{R}^{2})$. For any admissible pair $%
(q,r)$, we have
\begin{equation*}
\left\Vert e^{it\mathcal{L}}\varphi\right\Vert_{L^{q}((-T,T);L^{r})}\lesssim
\left\Vert\varphi\right\Vert_{L^{2}}.
\end{equation*}
$(ii)$ Let $I\subset (-T,T)$ be an interval and $t_{0}\in \overline{I}$. If $%
F\in L^{\tilde{q}^{\prime}}(I;L^{\tilde{r}^{\prime}})$, then for any
admissible pairs $(q,r)$ and $(\tilde{q},\tilde{r})$, we have
\begin{equation*}
\left\Vert \int_{t_{0}}^{t}e^{i(t-s)\mathcal{L}}F(s)ds\right%
\Vert_{L^{q}(I;L^{r})}\lesssim \left\Vert F\right\Vert_{L^{\tilde{q}%
^{\prime}}(I;L^{\tilde{r}^{\prime}})}.
\end{equation*}
\end{lemma}

\begin{lemma}
(\cite{M1}) Let $W$ be an arbitrary weight function such that $\nabla W,
\Delta W\in L^{\infty}(\mathbb{R}^{2})$. Then for all $T>0$ and admissible
pair $(q,r)$, we have the following estimates:\newline
$(i)$ $\left\Vert [\nabla, e^{it\mathcal{L}}]\varphi\right%
\Vert_{L^{q}((-T,T);L^{r})}\lesssim |T|\left\Vert \varphi\right\Vert_{L^{2}}$%
; \quad $(ii)$ $\left\Vert [W, e^{it\mathcal{L}}]\varphi\right%
\Vert_{L^{q}((-T,T);L^{r})}\lesssim |T|\left\Vert
(1+\nabla)\varphi\right\Vert_{L^{2}}$.
\end{lemma}

\begin{lemma}\label{L34}
(\cite{M1}) Let
\begin{equation*}
K(x,y)=\frac{\ln(|x-y|)-\ln(1+|x|)}{1+\ln(1+|y|)}
\end{equation*}
for $x,y\in \mathbb{R}^{2}$. For any $p\in[1,\infty)$ and $\varepsilon>0$
there exist a function $W(x,y)\geq0$ with $\|W\|_{L^{\infty}_{y}L^{p}_{x}}%
\leq\varepsilon$ and a constant $C$ such that
\begin{equation*}
|K(x,y)|\leq C+W(x,y)
\end{equation*}
for all $(x,y)\in\mathbb{R}^{2+2}$.
\end{lemma}

Recall that the following Logarithmic inequality. For more details of its proof,
we refer to \cite{CIMN,H,IMM}.

\begin{lemma}
\label{L50} Let $\beta \in (0,1)$. For any $\delta >\frac{1}{2\pi \beta }$
and any $0<\zeta \leq 1$, a constant $C_{\zeta }>0$ exists such that, for
any function $u\in H^{1}(\mathbb{R}^{2})\cap C^{\beta }(\mathbb{R}^{2})$, we
have
\begin{equation*}
\Vert u\Vert _{L^{\infty }}^{2}\leq \zeta \Vert u\Vert _{\zeta }^{2}\log
\left( C_{\zeta }+\frac{8^{\beta }\Vert u\Vert _{C^{\beta }}}{\zeta ^{\beta
}\Vert u\Vert _{\zeta }}\right) ,
\end{equation*}%
where
\begin{equation*}
\Vert u\Vert _{\zeta }^{2}=\Vert \nabla u\Vert _{L^{2}}^{2}+\zeta ^{2}\Vert
u\Vert _{L^{2}}^{2},
\end{equation*}%
and $C^{\beta }$ denotes the space of $\beta $-H\"{o}lder continuous
functions endowed with the norm
\begin{equation*}
\Vert u\Vert _{C^{\beta }}=\Vert u\Vert _{L^{\infty }}+\sup_{x\neq y}\frac{%
|u(x)-u(y)|}{|x-y|^{\beta }}.
\end{equation*}
\end{lemma}

\textbf{We are ready to prove Theorem \ref{T5}:} Define a Banach space
\begin{equation*}
\mathcal{H}_{R}:=\left\{ \psi \in L_{t}^{\infty }([0,T);X)\text{ }|\text{ } \Vert \psi \Vert
_{\mathcal{H}}\leq R\right\}
\end{equation*}%
with norm
\begin{equation*}
\Vert \psi \Vert _{\mathcal{H}}:=\Vert \psi \Vert _{L_{t}^{\infty }X}+\Vert
f\Vert _{L_{t}^{4}W^{1,4}}+\left\Vert \sqrt{\ln (1+|x|)}\psi \right\Vert
_{L_{t}^{4}L^{4}}.
\end{equation*}
We try to solve the following integral equation
\begin{equation*}
\psi (t)=e^{it\mathcal{L}}\psi _{0}+i\int_{0}^{t}e^{i(t-s)\mathcal{L}%
}\left[f(\psi )-\frac{\gamma }{2\pi }\psi \int_{\mathbb{R}^{2}} \ln \left(\frac{|x-y|}{1+|x|}\right) |\psi (y)|^{2}dy\right]ds.
\end{equation*}
Let
\begin{equation*}
\Psi (\psi (t))=e^{it\mathcal{L}}\psi _{0}+i\int_{0}^{t}e^{i(t-s)\mathcal{L}}\left[ f(\psi )-\frac{\gamma \psi }{2\pi }\int_{\mathbb{R}^{2}} \ln
\left(\frac{|x-y|}{1+|x|}\right) |\psi (y)|^{2}dy\right] ds=\Psi _{1}(\psi
(t))+\Psi _{2}(\psi (t)),
\end{equation*}%
where
\begin{equation*}
\Psi _{1}(\psi (t)):=\frac{1}{2}e^{it\mathcal{L}}\psi _{0}-\frac{i\gamma }{%
2\pi }\int_{0}^{t}e^{i(t-s)\mathcal{L}}\psi \int_{\mathbb{R}^{2}} \ln\left(
\frac{|x-y|}{1+|x|}\right) |\psi (y)|^{2}dyds,
\end{equation*}%
and
\begin{equation*}
\Psi_{2}(\psi(t)):=\frac{1}{2}e^{it\mathcal{L}}\psi
_{0}+i\int_{0}^{t}e^{i(t-s)\mathcal{L}}f(\psi )ds.
\end{equation*}
By Lemma \ref{L34}, there exist a nonnegative function $W\in L_{y}^{\infty }L_{x}^{\frac{4}{%
3}}$ and a positive constant $C$ such that
\begin{equation*}
|K(x,y)|\leq C+W(x,y).
\end{equation*}%
Then we have
\begin{equation*}
P\psi =\int K(x,y)(1+\ln (1+|y|))|\psi (y)|^{2}\psi (x)dy
\end{equation*}%
and
\begin{equation*}
\Vert P\psi \Vert _{L^{2}}\lesssim (\Vert \psi \Vert _{L^{2}}+\Vert \psi
\Vert _{L^{4}})\Vert \sqrt{1+\ln (1+|x|)}\psi \Vert _{L^{2}}^{2}.
\end{equation*}%
It follows from that
\begin{equation*}
\Vert P\psi \Vert _{L_{t}^{1}L^{2}}\lesssim (T\Vert \psi \Vert
_{L_{t}^{\infty }L^{2}}+T^{\frac{3}{4}}\Vert \psi \Vert
_{L_{t}^{4}L^{4}})\Vert \sqrt{1+\ln (1+|x|)}\psi \Vert _{L_{t}^{\infty
}L^{2}}^{2}.
\end{equation*}%
By Lemma \ref{L33}, we obtain
\begin{equation*}
\Vert \Psi _{1}(\psi )\Vert _{L_{t}^{\infty }L^{2}}+\Vert \Psi _{1}(\psi
)\Vert _{L_{t}^{4}L^{4}}\lesssim \Vert \psi _{0}\Vert _{L^{2}}+(T+T^{\frac{3%
}{4}})\Vert \psi \Vert _{\mathcal{H}}^{3}.
\end{equation*}%
Next, we estimate $\nabla \Psi _{1}(\psi )$. Note that
\begin{eqnarray*}
\nabla \Psi _{1}(\psi ) &=&e^{it\mathcal{L}}\nabla \psi _{0}-\frac{i\gamma }{%
2\pi }\int_{0}^{t}e^{i(t-s)\mathcal{L}}\nabla \left( \psi \int_{\mathbb{R}%
^{2}}\ln \left( \frac{|x-y|}{1+|x|}\right) |\psi (y)|^{2}dy\right) ds \\
&&+[\nabla ,e^{it\mathcal{L}}]\psi _{0}-\frac{i\gamma }{2\pi }%
\int_{0}^{t}[\nabla ,e^{i(t-s)\mathcal{L}}]\left( \psi \int_{\mathbb{R}%
^{2}}\ln \left( \frac{|x-y|}{1+|x|}\right) |\psi (y)|^{2}dy\right) ds.
\end{eqnarray*}%
We obtain
\begin{equation*}
\int_{0}^{t}\left\Vert [\nabla ,e^{i(t-s)\mathcal{L}}]P(\psi )\right\Vert
_{L^{2}}ds\leq \int_{0}^{t}(t-s)\left\Vert P(\psi )\right\Vert
_{L^{2}}ds\leq |t|\left\Vert P(\psi )\right\Vert _{L_{t}^{1}L^{2}}
\end{equation*}%
and
\begin{equation*}
\Vert P\nabla \psi \Vert _{L_{t}^{1}L^{2}}\lesssim (T\Vert \nabla \psi \Vert
_{L_{t}^{\infty }L^{2}}+T^{\frac{3}{4}}\Vert \nabla \psi \Vert
_{L_{t}^{4}L^{4}})\Vert \sqrt{1+\ln (1+|x|)}\psi \Vert _{L_{t}^{\infty
}L^{2}}^{2}.
\end{equation*}%
Since
\begin{equation*}
(\nabla P)\psi =\left( \int_{\mathbb{R}^{2}}\left( \frac{1}{|x-y|}-\frac{1}{%
1+|x|}\right) |\psi (y)|^{2}dy\right) \psi (x),
\end{equation*}%
it follows from Hardy-Littlewood-Sobolev and Sobolev inequalities that
\begin{eqnarray*}
\Vert (\nabla P)\psi \Vert _{L^{2}} &\lesssim &\left\Vert (|x|^{-1}\ast
|\psi |^{2})+(1+|x|)^{-1}\Vert \psi \Vert _{L^{2}}^{2}\right\Vert
_{L^{4}}\Vert \psi \Vert _{L^{4}} \\
&\lesssim &\left( \Vert \psi \Vert _{L^{\frac{8}{3}}}^{2}+\Vert \psi \Vert
_{L^{2}}^{2}\right) \Vert \psi \Vert _{L^{4}} \\
&\lesssim &\left( \Vert \nabla \psi \Vert _{L^{2}}^{2}+\Vert \psi \Vert
_{L^{2}}^{2}\right) \Vert \psi \Vert _{L^{4}}.
\end{eqnarray*}%
Thus we have
\begin{equation*}
\Vert (\nabla P)\psi \Vert _{L_{t}^{1}L^{2}}\lesssim T^{\frac{3}{4}}\left(
\Vert \nabla \psi \Vert _{L_{t}^{\infty }L^{2}}^{2}+\Vert \psi \Vert
_{L_{t}^{\infty }L^{2}}^{2}\right) \Vert \psi \Vert _{L_{t}^{4}L^{4}}.
\end{equation*}%
By Lemma \ref{L33}, we deduce that
\begin{equation*}
\Vert \nabla \Psi _{1}(\psi )\Vert _{L_{t}^{\infty }L^{2}}+\Vert \nabla \Psi_{1}
\Vert _{L_{t}^{4}L^{4}}\lesssim \Vert \nabla \psi _{0}\Vert _{X}+(T+T^{\frac{%
3}{4}})\Vert \psi \Vert _{\mathcal{H}}^{3}.
\end{equation*}%
Let us give the estimate of $\sqrt{\ln (1+|x|)}\Psi _{1}(\psi )$. It
holds
\begin{equation*}
\sqrt{1+\ln (1+|x|)}\Psi_{1}(\psi )=e^{it\mathcal{L}}\sqrt{1+\ln (1+|x|)}%
\psi _{0}+\frac{i\gamma }{2\pi }\int_{0}^{t}e^{i(t-s)\mathcal{L}}\sqrt{1+\ln
(1+|x|)}P\psi ds+M_{1},
\end{equation*}%
where
\begin{equation*}
M_{1}:=[\sqrt{1+\ln (1+|x|)},e^{it\mathcal{L}}]\psi _{0}+\frac{i\gamma }{2\pi
}\int_{0}^{t}[\sqrt{1+\ln (1+|x|)},e^{i(t-s)\mathcal{L}}]P\psi ds.
\end{equation*}%
Let $W=\sqrt{1+\ln (1+|x|)}$, we have
\begin{eqnarray*}
\Vert M_{1}\Vert _{L_{t}^{\infty }L^{2}}+\Vert M_{1}\Vert _{L_{t}^{4}L^{4}}
&\lesssim &T\Vert \psi _{0}\Vert _{X}+T\Vert (1+\nabla )(P\psi )\Vert
_{L_{t}^{1}L^{2}} \\
&\lesssim &T\Vert \psi _{0}\Vert _{X}+T(T+T^{\frac{3}{4}})\Vert \psi \Vert _{%
\mathcal{H}}^{3},
\end{eqnarray*}%
and
\begin{eqnarray*}
\Vert P(W\psi )\Vert _{L_{t}^{1}L^{2}} &\lesssim &(T\Vert W\psi \Vert
_{L_{t}^{\infty }L^{2}}+T^{\frac{3}{4}}\Vert W\psi \Vert
_{L_{t}^{4}L^{4}})\Vert W\psi \Vert _{L_{t}^{\infty }L^{2}} \\
&\lesssim &(T+T^{\frac{3}{4}})\Vert \psi \Vert _{\mathcal{H}}^{3}.
\end{eqnarray*}

Finally, we estimate $\Psi_{2}(\psi)$. By Lemma \ref{L33}, we can deduce that
\begin{equation*}
\Vert \Psi_{2}(\psi )\Vert _{L_{t}^{\infty }H^{1}}\lesssim \Vert \psi
_{0}\Vert _{X}+\Vert f(\psi )\Vert _{L_{t}^{1}H^{1}}.
\end{equation*}%
By the H\"{o}lder inequality, for any $\varepsilon >0$, we have
\begin{equation*}
\left\Vert e^{4\pi (1+\varepsilon )|\psi |^{2}}-1\right\Vert _{L_{t}^{\frac{4%
}{3}}L^{4}}\lesssim \Vert e^{3\pi (1+\varepsilon )|\psi |_{L^{\infty
}}^{2}}\Vert _{L_{t}^{\frac{4}{3}}}\Vert e^{4\pi (1+\varepsilon )|\psi
|^{2}}-1\Vert _{L_{t}^{\infty }L^{1}}^{\frac{1}{4}}.
\end{equation*}%
Let $\Vert \nabla \psi \Vert _{L_{t}^{\infty }L^{2}}^{2}<1$ and take $%
\varepsilon >0$ small such that
\begin{equation*}
(1+\varepsilon )\Vert \nabla \psi \Vert _{L_{t}^{\infty }L^{2}}<1.
\end{equation*}%
By using Moser-Trudinger inequality, we have
\begin{equation*}
\int_{\mathbb{R}^{2}}(e^{4\pi (1+\varepsilon )|\psi |^{2}})dx\lesssim \int_{%
\mathbb{R}^{2}}(e^{4\pi (1+\varepsilon )|\nabla \psi |_{L_{t}^{\infty
}L^{2}}^{2}(\frac{|\psi |}{|\nabla \psi |})^{2}})dx\lesssim \Vert \psi \Vert
_{L^{2}}^{2}\lesssim 1.
\end{equation*}%
By Lemma \ref{L50}, for any $\delta >\frac{1}{\pi }$ and $0<\zeta \leq 1$,
we get
\begin{equation*}
e^{4\pi (1+\varepsilon )\Vert \psi \Vert _{L^{\infty }}}\leq \left( C+2\sqrt{%
\frac{2}{\zeta }}\frac{\Vert \psi \Vert _{C^{\frac{1}{2}}}}{\Vert \psi \Vert
_{\zeta }}\right) ^{\delta 4\pi (1+\varepsilon )\Vert \psi \Vert _{\zeta
}^{2}}.
\end{equation*}%
Since
\begin{equation*}
\Vert \psi \Vert _{\zeta }^{2}=\zeta ^{2}\Vert \psi \Vert _{L^{2}}^{2}+\Vert
\nabla \psi \Vert _{L^{2}}^{2}<1+\zeta ^{2}\Vert \psi \Vert _{L_{t}^{\infty
}H^{1}}^{2},
\end{equation*}%
we may take $0<\zeta ,\varepsilon $ near to zero, and $0<\alpha <4$ near to $%
4$ such that $4\pi (1+\varepsilon )\Vert \psi \Vert _{\zeta }^{2}<4\pi $.
Thus for $\delta >\frac{1}{\pi }$ near to $\frac{1}{\pi }$,
\begin{equation*}
e^{4\pi (1+\varepsilon )\Vert \psi \Vert _{L^{\infty }}}\lesssim (1+\Vert
\psi \Vert _{C^{\frac{1}{2}}})^{\alpha }\lesssim 1+\Vert \psi \Vert
_{W^{1,4}}^{\alpha }.
\end{equation*}%
So we have
\begin{eqnarray*}
\left\Vert e^{4\pi (1+\varepsilon )|\psi |^{2}}-1\right\Vert _{L_{t}^{\frac{4%
}{3}}L^{4}} &\lesssim &\Vert e^{3\pi (1+\varepsilon )\Vert \psi \Vert
_{L^{\infty }}^{2}}\Vert _{L_{t}^{\frac{4}{3}}}\Vert e^{(1+\varepsilon
)|\psi |^{2}}-1\Vert _{L_{t}^{\infty }L^{1}}^{\frac{1}{4}} \\
&\lesssim &\Vert e^{3\pi (1+\varepsilon )\Vert \psi \Vert _{L^{\infty
}}^{2}}\Vert _{L_{t}^{\frac{4}{3}}} \\
&\lesssim &\Vert 1+\Vert \psi \Vert _{W^{1,4}}^{\alpha }\Vert _{L_{t}^{1}}^{%
\frac{3}{4}} \\
&\lesssim &T^{\frac{3}{4}}+T^{\frac{3}{4}(1-\frac{\alpha }{4})}\Vert \psi
\Vert _{L_{t}^{4}W^{1,4}}^{\frac{3\alpha }{4}}.
\end{eqnarray*}%
By condition $(f_{7})$, H\"{o}lder inequality and above estimates, we have
\begin{eqnarray*}
\Vert f(\psi )\Vert _{L_{t}^{1}L^{2}} &\lesssim &\left\Vert \psi (e^{4\pi
(1+\varepsilon )|\psi |^{2}}-1)\right\Vert _{L_{t}^{1}L^{2}} \\
&\lesssim &\left\Vert \psi \right\Vert _{L_{t}^{4}L^{4}}\left\Vert e^{4\pi
(1+\varepsilon )|\psi |^{2}}-1\right\Vert _{L_{t}^{\frac{4}{3}}L^{4}} \\
&\lesssim &\left\Vert \psi \right\Vert _{\mathcal{H}}\left( T^{\frac{3}{4}%
}+\left\Vert \psi \right\Vert _{\mathcal{H}}^{\frac{3}{4}\alpha }T^{\frac{3}{%
4}(1-\frac{\alpha }{4})}\right) ,
\end{eqnarray*}%
and
\begin{eqnarray*}
\Vert \nabla f(\psi )\Vert _{L_{t}^{1}L_{x}^{2}} &\lesssim &\left\Vert
\nabla \psi (|\psi |+e^{4\pi (1+\varepsilon )|\psi |^{2}}-1)\psi \right\Vert
_{L_{t}^{1}L^{2}} \\
&\lesssim &\left\Vert \nabla \psi |\psi |^{2}\right\Vert
_{L_{t}^{1}L^{2}}+\left\Vert \nabla \psi (e^{4\pi (1+\varepsilon )|\psi
|^{2}}-1)\psi \right\Vert _{L_{t}^{1}L^{2}} \\
&\lesssim &\left\Vert \psi \right\Vert _{L_{t}^{\infty }H^{1}}\left\Vert
\nabla \psi \right\Vert _{L_{t}^{4}L^{4}}\left( \left\Vert \psi \right\Vert
_{L_{t}^{\infty }H^{1}}T^{\frac{3}{4}}+\left\Vert e^{4\pi (1+\varepsilon
)|\psi |^{2}}-1\right\Vert _{L_{t}^{\frac{4}{3}}L^{4+\varepsilon }}\right) \\
&\lesssim &\left\Vert \psi \right\Vert _{\mathcal{H}}^{2}\left( \left\Vert
\psi \right\Vert _{\mathcal{H}}T^{\frac{3}{4}}+T^{\frac{3}{4}}+\left\Vert
\psi \right\Vert _{\mathcal{H}}^{\frac{3}{4}\alpha }T^{\frac{3}{4}(1-\frac{%
\alpha }{4})}\right) .
\end{eqnarray*}%
Then we get
\begin{equation*}
\Vert \Psi _{2}(\psi )\Vert _{L_{t}^{\infty }H^{1}}\lesssim \Vert \psi
_{0}\Vert _{X}+\left\Vert \psi \right\Vert _{\mathcal{H}}\left( T^{\frac{3}{4%
}}+\left\Vert \psi \right\Vert _{\mathcal{H}}^{\frac{3}{4}\alpha }T^{\frac{3%
}{4}(1-\frac{\alpha }{4})}\right) +\left\Vert \psi \right\Vert _{\mathcal{H}%
}^{2}\left( \left\Vert \psi \right\Vert _{\mathcal{H}}T^{\frac{3}{4}}+T^{%
\frac{3}{4}}+\left\Vert \psi \right\Vert _{\mathcal{H}}^{\frac{3}{4}\alpha
}T^{\frac{3}{4}(1-\frac{\alpha }{4})}\right) .
\end{equation*}%
Let us estimate the term $\sqrt{\ln (1+|x|)}\Psi_{2}(\psi )$. It
holds
\begin{equation*}
\sqrt{1+\ln (1+|x|)}\Psi_{2}(\psi )=e^{it\mathcal{L}}\sqrt{1+\ln (1+|x|)}%
\psi _{0}+i\int_{0}^{t}e^{i(t-s)\mathcal{L}}\sqrt{1+\ln (1+|x|)}f(\psi
)ds+M_{2},
\end{equation*}%
where
\begin{equation*}
M_{2}:=[\sqrt{1+\ln (1+|x|)},e^{it\mathcal{L}}]\psi _{0}+i\int_{0}^{t}[\sqrt{%
1+\ln (1+|x|)},e^{i(t-s)\mathcal{L}}]f(\psi )ds.
\end{equation*}%
Let $W=\sqrt{1+\ln (1+|x|)}$, we have
\begin{eqnarray*}
\Vert M_{2}\Vert _{L_{t}^{\infty }L^{2}}+\Vert M_{2}\Vert _{L_{t}^{4}L^{4}}
&\lesssim &T\Vert \psi _{0}\Vert _{X}+T\Vert (1+\nabla )(f(\psi ))\Vert
_{L_{t}^{1}L^{2}} \\
&\lesssim &T\left\Vert \psi \right\Vert _{\mathcal{H}}\left( T^{\frac{3}{4}%
}+\left\Vert \psi \right\Vert _{\mathcal{H}}^{\frac{3}{4}\alpha }T^{\frac{3}{%
4}(1-\frac{\alpha }{4})}\right)\\
&&+\left\Vert \psi \right\Vert _{\mathcal{H}%
}^{2}\left( \left\Vert \psi \right\Vert _{\mathcal{H}}T^{\frac{3}{4}}+T^{%
\frac{3}{4}}+\left\Vert \psi \right\Vert _{\mathcal{H}}^{\frac{3}{4}\alpha
}T^{\frac{3}{4}(1-\frac{\alpha }{4})}\right) ,
\end{eqnarray*}%
and
\begin{eqnarray*}
\Vert f(\psi )W\Vert _{L_{t}^{1}L^{2}} &\lesssim &\left\Vert \psi \sqrt{%
1+\ln (1+|x|)}(e^{4\pi (1+\varepsilon )|\psi |^{2}}-1)\right\Vert
_{L_{t}^{1}L^{2}} \\
&\lesssim &\left\Vert \psi \sqrt{1+\ln (1+|x|)}\right\Vert
_{L_{t}^{4}L^{4}}\left\Vert e^{4\pi (1+\varepsilon )|\psi
|^{2}}-1\right\Vert _{L_{t}^{\frac{4}{3}}L^{4}} \\
&\lesssim &\left\Vert \psi \right\Vert _{\mathcal{H}}\left( T^{\frac{3}{4}%
}+\left\Vert \psi \right\Vert _{\mathcal{H}}^{\frac{3}{4}\alpha }T^{\frac{3}{%
4}(1-\frac{\alpha }{4})}\right).
\end{eqnarray*}
We can conclude that
\begin{eqnarray*}
\Vert \Psi (\psi )\Vert _{\mathcal{H}} &\lesssim &\Vert \psi _{0}\Vert
_{X}+(T+T^{\frac{3}{4}})\Vert \psi \Vert _{\mathcal{H}}^{3}+\left\Vert \psi
\right\Vert _{\mathcal{H}}\left( T^{\frac{3}{4}}+\left\Vert \psi \right\Vert
_{\mathcal{H}}^{\frac{3}{4}\alpha }T^{\frac{3}{4}(1-\frac{\alpha }{4}%
)}\right) \\
&&+\left\Vert \psi \right\Vert _{\mathcal{H}}^{2}\left( \left\Vert \psi
\right\Vert _{\mathcal{H}}T^{\frac{3}{4}}+T^{\frac{3}{4}}+\left\Vert \psi
\right\Vert _{\mathcal{H}}^{\frac{3}{4}\alpha }T^{\frac{3}{4}(1-\frac{\alpha
}{4})}\right) .
\end{eqnarray*}%
Thus, if we take $R\geq 2\Vert \psi _{0}\Vert _{X}$, then there exists $T>0$
such that $\Psi$ is a contraction map from $\mathcal{H}_{R}$ to itself. A
similar argument shows that $\Psi$ has a unique fixed point in this space.

\section{The ground state standing waves}

\begin{lemma}
\label{L3.0} For any $\rho >0,$ there holds $S(c)\cap \mathcal{B}_{\rho
}\neq \emptyset $ if $0<c<\rho,$ where $\mathcal{B}_{\rho }$ is as in (\ref%
{e1-6}).
\end{lemma}

\begin{proof}
Let $\phi (x)=e^{-\frac{1}{2}|x|^{2}}.$ Then a direct calculation shows that
\begin{equation*}
\Vert \nabla \phi \Vert _{L^{2}}^{2}=\Vert \phi \Vert _{L^{2}}^{2}=\pi .
\end{equation*}%
Then for $0<c<\rho ,$ we have $\bar{\phi}:=\sqrt{\frac{c}{\pi }}\phi \in
S(c)\cap \mathcal{B}_{\rho }.$ We complete the proof.
\end{proof}

\begin{lemma}
\label{L3.1} Assume that conditions $(f_{1})$ and $(f_{5})$ hold and $%
F(t)\geq 0$ for $t>0.$ Then for any $0<\rho <1,$ we have%
\begin{equation*}
\gamma _{c}^{\rho }=\inf_{u\in S(c)\cap \mathcal{B}_{\rho }}J(u)>-\infty
\text{ if }0<c<\rho.  \label{LML}
\end{equation*}%
Furthermore, there holds
\begin{equation*}
\gamma _{c}^{\rho }\leq \frac{1}{2}c+\frac{\sqrt{\pi }c^{3}}{4}.
\label{LML1}
\end{equation*}
\end{lemma}

\begin{proof}
By conditions $(f_{1})$ and $(f_{5}),$ for any $\xi >0$ and for fixed $q>2,$
there exists a constant $C_{1}=C_{1}(\xi ,\alpha ,q)>0$ such that%
\begin{equation*}
F(t)\leq \xi |t|^{2}+C_{1}|t|^{q}(e^{\alpha t^{2}}-1),\text{ }\forall t\in
\mathbb{R},
\end{equation*}%
which shows that%
\begin{equation}
\int_{\mathbb{R}^{2}}F(u)dx\leq \xi \int_{\mathbb{R}^{2}}|u|^{2}dx+C_{1}%
\int_{\mathbb{R}^{2}}|u|^{q}(e^{\alpha u^{2}}-1)dx.  \label{e3-6}
\end{equation}%
Let $u\in S(c)\cap \mathcal{B}_{\rho }$ with $0<\rho <1.$ Then we can choose
$\alpha >4\pi $ close to $4\pi $ and $\bar{\eta}>1$ close to $1$ such that $%
\alpha \bar{\eta}\rho <4\pi .$ By Lemma \ref{L2-7}, the H\"{o}lder
inequality and the fact that%
\begin{equation*}
(e^{s}-1)^{t}\leq e^{st}-1\text{ for }s\geq 0\text{ and }t>1,
\end{equation*}%
we have%
\begin{eqnarray}
\int_{\mathbb{R}^{2}}|u|^{q}(e^{\alpha u^{2}}-1)dx &\leq &\left( \int_{%
\mathbb{R}^{2}}|u|^{q\eta }dx\right) ^{1/\eta }\left( \int_{\mathbb{R}%
^{2}}(e^{\alpha u^{2}}-1)^{\bar{\eta}}dx\right) ^{1/\bar{\eta}}  \notag \\
&\leq &\Vert u\Vert _{L^{q\eta }}^{q}\left[ \int_{\mathbb{R}^{2}}\left(
e^{\alpha \bar{\eta}\rho \left( u/\Vert \nabla u\Vert_{L^{2}}\right)
^{2}}-1\right) dx\right] ^{1/\bar{\eta}}  \notag \\
&\leq &C_{2}\Vert u\Vert _{L^{q\eta }}^{q}  \label{e3-7}
\end{eqnarray}%
for some $C_{2}=C_{2}(\alpha ,\bar{\eta},\rho )>0,$ where $\eta :=\frac{\bar{%
\eta}}{\bar{\eta}-1}>1.$ Thus, it follows from (\ref{e3-6}) and (\ref{e3-7})
that there exists a constant $K_{1}=C_{1}C_{2}>0$ such that%
\begin{equation}
\int_{\mathbb{R}^{2}}F(u)dx\leq \xi \Vert u\Vert _{L^{2}}^{2}+K_{1}\Vert
u\Vert _{L^{q\eta}}^{q}.  \label{e3-4}
\end{equation}%
Since $V_{1}(u)\geq 0$, by (\ref{GN}), (\ref{V2}) and (\ref{e3-4}), we have
\begin{eqnarray}
J(u) &=&\frac{1}{2}A(u)+\frac{1}{4}V(u)-\int_{\mathbb{R}^{2}}F(u)dx  \notag
\\
&\geq &\frac{1}{2}A(u)-\frac{1}{4}V_{2}(u)-\xi \Vert u\Vert
_{L^{2}}^{2}-K_{1}\Vert u\Vert _{L^{q\eta }}^{q}  \notag \\
&\geq &\frac{1}{2}A(u)-\frac{1}{4}Kc^{\frac{3}{2}}A(u)^{\frac{1}{2}}-\xi
c-K_{1}(c\mathcal{S}_{q\eta })^{\frac{1}{\eta }}A(u)^{\frac{q\eta -2}{2\eta }%
},  \label{e3-1}
\end{eqnarray}%
which implies that $\gamma _{c}^{\rho }=\inf_{u\in S(c)\cap \mathcal{B}%
_{\rho }}J(u)>-\infty .$

Moreover, it follows from (\ref{V1}) and the fact of $F(t)\geq 0$ for $t>0$
that%
\begin{eqnarray}
J(\bar{\phi}) &\leq &\frac{1}{2}A(\bar{\phi})+\frac{1}{4}V_{1}(\bar{\phi})
\label{e3-5} \\
&=&\frac{1}{2}A(\bar{\phi})+\frac{1}{2}\Vert \bar{\phi}\Vert _{\ast
}^{2}c^{2}  \notag \\
&\leq &\frac{1}{2}c+\frac{\sqrt{\pi }c^{3}}{4}.  \notag
\end{eqnarray}%
Hence, we have
\begin{equation*}
\gamma _{c}^{\rho }\leq \frac{1}{2}c+\frac{\sqrt{\pi }c^{3}}{4}.
\end{equation*}%
We complete the proof.
\end{proof}

\begin{lemma}
\label{L3.2} Assume that conditions $(f_{1})$ and $(f_{5})$ hold and $%
F(t)\geq 0$ for $t>0.$ For any $0<\rho <1,$ if $S(c)\cap (\mathcal{B}_{\rho
}\backslash \mathcal{B}_{b\rho })\neq \emptyset ,$ then there exists $%
c_{1}>0 $ such that for any $0<c<c_{1}$,
\begin{equation*}
\inf_{u\in S(c)\cap \mathcal{B}_{a\rho }}J(u)<\inf_{u\in S(c)\cap (\mathcal{B%
}_{\rho }\backslash \mathcal{B}_{b\rho })}J(u),
\end{equation*}%
where $0<a<b<1$.
\end{lemma}

\begin{proof}
For any $\rho >0,$ it follows from Lemma \ref{L3.0} that $S(c)\cap \mathcal{B%
}_{a\rho }\neq \emptyset $ for $0<c<a\rho .$ By (\ref{e3-5}), for $0<c<a\rho
,$ we have
\begin{eqnarray*}
J(\bar{\phi}) &\leq &\frac{1}{2}A(\bar{\phi})+\frac{1}{4}V_{1}(\bar{\phi}) \\
&<&\frac{1}{2}a\rho +\frac{\sqrt{\pi }c^{3}}{4},
\end{eqnarray*}%
which implies that
\begin{equation}
\inf_{u\in S(c)\cap \mathcal{B}_{a\rho }}J(u)<\frac{1}{2}a\rho +\frac{\sqrt{%
\pi }c^{3}}{4}.  \label{3.1.1}
\end{equation}%
On the other hand, for any $u\in \mathcal{B}_{\rho }\backslash \mathcal{B}%
_{b\rho }$, by (\ref{e3-1}) one has
\begin{eqnarray*}
J(u) &\geq &\frac{1}{2}A(u)-\frac{1}{4}Kc^{\frac{3}{2}}A(u)^{\frac{1}{2}%
}-\xi c-K_{1}(c\mathcal{S}_{q\eta })^{\frac{1}{\eta }}A(u)^{\frac{q\eta -2}{%
2\eta }} \\
&\geq &\frac{1}{2}b\rho -\frac{\sqrt{\rho }}{4}Kc^{\frac{3}{2}}-\xi c-K_{1}(c%
\mathcal{S}_{q\eta })^{\frac{1}{\eta }}\rho ^{\frac{q\eta -2}{2\eta }},
\end{eqnarray*}%
leading to%
\begin{equation}
\inf_{u\in S(c)\cap (\mathcal{B}_{\rho }\backslash \mathcal{B}_{b\rho
})}J(u)\geq \frac{1}{2}b\rho -\frac{\sqrt{\rho }}{4}Kc^{\frac{3}{2}}-\xi
c-K_{1}(c\mathcal{S}_{q\eta })^{\frac{1}{\eta }}\rho ^{\frac{q\eta -2}{2\eta
}}.  \label{3.1.2}
\end{equation}%
We note that there exists a constant $c_{0}=c_{0}(a,b,\rho )>0$ such that
for $0<c<c_{0},$%
\begin{equation}
\frac{\sqrt{\pi }}{4}c^{3}+\frac{\sqrt{\rho }K}{4}c^{\frac{3}{2}}+\xi
c+K_{1}(c\mathcal{S}_{q\eta })^{\frac{1}{\eta }}\rho ^{\frac{q\eta -2}{2\eta
}}<\frac{1}{2}(b-a)\rho .  \label{3.1.5}
\end{equation}%
Thus, it follows from (\ref{3.1.1})--(\ref{3.1.5}) that for $0<c<c_{1}:=\min
\{a\rho ,c_{0}\},$
\begin{equation*}
\inf_{u\in S(c)\cap \mathcal{B}_{a\rho }}J(u)<\inf_{u\in S(c)\cap (\mathcal{B%
}_{\rho }\backslash \mathcal{B}_{b\rho })}J(u).
\end{equation*}%
We complete the proof.
\end{proof}

\textbf{Now we are ready to prove Theorem \ref{T1}: }For any $0<\rho <1,$
let $\{u_{n}\}\subset S(c)\cap \mathcal{B}_{\rho }$ be a minimizing sequence
for $\gamma _{c}^{\rho }$. It follows from Lemma \ref{L2.6} that $\{u_{n}\}$
is bounded in $X$ for $0<c<1-\rho $. Up to a subsequence, we can assume that
$u_{n}\rightharpoonup u_{c}$ in $X.$ From Lemma \ref{L2.1}$(i)$, it follows
that $u_{c}\in S(c)$. Moreover, we have
\begin{equation*}
A(u_{c})\leq \liminf_{n\rightarrow \infty }A(u_{n})\leq \rho ,
\end{equation*}%
leading to $u_{c}\in \mathcal{B}_{\rho }$. Hence, we have $u_{c}\in S(c)\cap
\mathcal{B}_{\rho }.$

By Lemma \ref{L2.1}$(iii)-(iv),$ we obtain that
\begin{equation}
\lim_{n\rightarrow \infty }V_{2}(u_{n})=V_{2}(u_{c})\text{ and }%
V_{1}(u_{c})\leq \liminf_{n\rightarrow \infty }V_{1}(u_{n}),  \label{e3-2}
\end{equation}%
respectively. For $0<c<1-\rho ,$ it follows from Lemma \ref{L2-10} that
\begin{equation}
\lim_{n\rightarrow \infty }\int_{\mathbb{R}^{2}}F(u_{n})dx=\int_{\mathbb{R}%
^{2}}F(u_{c})dx.  \label{e3-3}
\end{equation}%
Using (\ref{e3-2}) and (\ref{e3-3}), gives
\begin{equation*}
\gamma _{c}^{\rho }=\lim_{n\rightarrow \infty }J(u_{n})\geq J(u_{c})\geq
\gamma _{c}^{\rho },
\end{equation*}%
which implies that $J(u_{c})=\gamma _{c}^{\rho }.$

Since $J(u_{n})\rightarrow J(u_{c})$ and $V_{2}(u_{n})\rightarrow
V_{2}(u_{c})$, together with (\ref{e3-3}) again, we get
\begin{equation}
\frac{1}{2}[A(u_{n})-A(u_{c})]+\frac{1}{4}[V_{1}(u_{n})-V_{1}(u_{c})]=o(1).
\label{3.1.3}
\end{equation}%
Taking the $\liminf $ in (\ref{3.1.3}), we have
\begin{equation*}
\frac{1}{2}[\liminf_{n\rightarrow \infty }A(u_{n})-A(u_{c})]+\frac{1}{4}%
[\liminf_{n\rightarrow \infty }V_{1}(u_{n})-V_{1}(u_{c})]\leq 0.
\end{equation*}%
And using the weak lower semi-continuity of $A(u)$ and $V_{1}(u)$, we deduce
that
\begin{equation*}
\liminf_{n\rightarrow \infty }A(u_{n})=A(u_{c})\ \text{and}\
\liminf_{n\rightarrow \infty }V_{1}(u_{n})=V_{1}(u_{c}).
\end{equation*}%
Similarly, taking the $\limsup $ in (\ref{3.1.3}), we get
\begin{equation*}
\limsup_{n\rightarrow \infty }A(u_{n})=A(u_{c})\ \text{and}\
\limsup_{n\rightarrow \infty }V_{1}(u_{n})=V_{1}(u_{c}).
\end{equation*}%
Hence, we obtain that $A(u_{n})\rightarrow A(u_{c})$ and $%
V_{1}(u_{n})\rightarrow V_{1}(u_{c})$. This shows that $u_{n}\rightarrow
u_{c}$ in $H^{1}(\mathbb{R}^{2})$.

Next, we claim that $\Vert u_{n}-u_{c}\Vert _{\ast }\rightarrow 0$ as $%
n\rightarrow \infty $. By Lemma \ref{L2.2}, we only need to prove that
\begin{equation}
B_{1}(u_{n}^{2},(u_{n}-u_{c})^{2})\rightarrow 0\text{ as }n\rightarrow
\infty .  \label{3.1.9}
\end{equation}%
Indeed, we have
\begin{equation}
B_{1}(u_{n}^{2},(u_{n}-u_{c})^{2})=V_{1}(u_{n})-2B_{1}(u_{n}^{2},(u_{n}-u_{c})u_{c})-B_{1}(u_{n}^{2},u_{c}^{2}).
\label{3.1.6}
\end{equation}%
Since $\{u_{n}\}$ is bounded in $X$ and $u_{n}\rightharpoonup u_{c}$ in $X$,
it follows from Lemma \ref{L2.3} that
\begin{equation}
B_{1}(u_{n}^{2},(u_{n}-u_{c})u_{c})\rightarrow 0\text{ as }n\rightarrow
\infty .  \label{3.1.7}
\end{equation}%
Since $u_{n}\rightarrow u_{c}$ a.e. in $\mathbb{R}^{2}$, using Fatou's Lemma
gives%
\begin{equation}
V_{1}(u_{c})\leq \liminf_{n\rightarrow \infty }B_{1}(u_{n}^{2},u_{c}^{2}).
\label{3.1.8}
\end{equation}%
Thus, by (\ref{3.1.6})--(\ref{3.1.8}) one has%
\begin{equation*}
\limsup_{n\rightarrow \infty }B_{1}(u_{n}^{2},(u_{n}-u_{c})^{2})\leq
\limsup_{n\rightarrow \infty }V_{1}(u_{n})-\liminf_{n\rightarrow \infty
}B_{1}(u_{n}^{2},u_{c}^{2})\leq \limsup_{n\rightarrow \infty
}V_{1}(u_{n})-V_{1}(u_{c}),
\end{equation*}%
which implies that (\ref{3.1.9}) holds, since $%
B_{1}(u_{n}^{2},(u_{n}-u_{c})^{2})\geq 0$ and $V_{1}(u_{n})\rightarrow
V_{1}(u_{c})$. Hence, $u_{n}\rightarrow u_{c}$ in $X.$

Finally, it follows from Lemma \ref{L3.2} that $u_{c}\not\in S(c)\cap
\partial \mathcal{B}_{\rho }$ as $u_{c}\in \mathcal{B}_{\rho }$, where $%
\partial \mathcal{B}_{\rho }:=\left\{ u\in X\text{ }|\text{\ }A(u)=\rho
\right\} $. Then $u_{c}$ is indeed a critical point of $J|_{S(c)}$.
Therefore, for $0<c<c_{\ast }:=\min \{c_{1},1-\rho \},$ there exists a
Lagrange multiplier $\lambda _{c}\in \mathbb{R}$ such that $(u_{c},\lambda
_{c})$ is a couple of weak solutions to problem $(SP_{c}).$ We complete the
proof.

\textbf{Next, we give the proof of Theorem \ref{T2}:} Motivated by \cite%
{BBJV}. On the contrary, we assume that there exists $v_{c}\in S(c)$ such
that
\begin{equation*}
J|_{S(c)}^{\prime }(v_{c})=0\ \text{and}\ J(v_{c})<\gamma _{c}^{\rho }.
\end{equation*}%
Then, $v_{c}$ is a weak solution of the equation
\begin{equation*}
-\Delta v_{c}+\bar{\lambda}v_{c}+(\log |\cdot |\ast v_{c}^{2})v_{c}=f(v_{c}),
\end{equation*}%
for some $\bar{\lambda}\in \mathbb{R}$. By Lemma \ref{L2.7}, we have
\begin{equation}
Q(v_{c})=A(v_{c})-\frac{1}{4}\Vert v_{c}\Vert _{L^{2}}^{4}+\int_{\mathbb{R}%
^{2}}(2F(v_{c})-f(v_{c})v_{c})dx=0.  \label{3.1.10}
\end{equation}%
Moreover, from (\ref{LML1}) it follows that%
\begin{equation}
J(v_{c})<\frac{1}{2}c+\frac{\sqrt{\pi }c^{3}}{4}.  \label{3.1.11}
\end{equation}%
Thus, by Lemma \ref{L2.5}, (\ref{3.1.10}) and (\ref{3.1.11}), we get
\begin{eqnarray}
\frac{1}{2}c+\frac{\sqrt{\pi }c^{3}}{4} &>&J(v_{c})  \notag \\
&\geq &\frac{1}{p-2}Q(v_{c})+\frac{p-4}{2(p-2)}A(v_{c})-\frac{K}{4}\Vert
v_{c}\Vert _{L^{2}}^{3}A(v_{c})^{1/2}+\frac{1}{4(p-2)}\Vert v_{c}\Vert
_{L^{2}}^{4}.  \notag \\
&=&\frac{p-4}{2(p-2)}A(v_{c})-\frac{K}{4}\Vert v_{c}\Vert
_{L^{2}}^{3}A(v_{c})^{1/2}+\frac{1}{4(p-2)}\Vert v_{c}\Vert _{L^{2}}^{4}
\notag \\
&\geq &\frac{p-4}{2(p-2)}A(v_{c})-\frac{K}{4}c^{3/2}A(v_{c})^{1/2}+\frac{1}{%
4(p-2)}c^{2},  \label{3.1.12}
\end{eqnarray}%
which implies that there exists $0<\bar{c}_{\ast }\leq c_{\ast }$ such that $%
A(v_{c})<\rho $ for $0<c<\bar{c}_{\ast }$. Hence, we have $v_{c}\in \mathcal{%
B}_{\rho }$ and thus $J(v_{c})\geq \gamma _{c}^{\rho }$, which contradicts
with $J(v_{c})<\gamma _{c}^{\rho }.$ This shows that $u_{c}$ is a ground
state of problem $(SP_{c})$ with $\bar{\lambda}\in \mathbb{R}$.

Finally, similar to (\ref{3.1.12}), we have $A(u_{c})\rightarrow 0$ as $%
c\rightarrow 0,$ and%
\begin{eqnarray}
\gamma _{c}^{\rho } &=&J(u_{c})\geq \frac{p-4}{2(p-2)}A(u_{c})-\frac{K}{4}%
c^{3/2}A(u_{c})^{1/2}+\frac{1}{4(p-2)}c^{2}  \notag \\
&\geq &-\frac{K^{2}(p-2)}{32(p-4)}c^{3}+\frac{1}{4(p-2)}c^{2}  \label{3.1.13}
\\
&>&0,  \notag
\end{eqnarray}%
provided that%
\begin{equation*}
0<c<\tilde{c}_{\ast }:=\min \left\{ c_{\ast },\frac{8(p-4)}{K^{2}(p-2)^{2}}%
\right\} .
\end{equation*}%
Moreover, by (\ref{LML1}) and (\ref{3.1.13}) one has $\gamma _{c}^{\rho
}\rightarrow 0$ as $c\rightarrow 0.$ We complete the proof.

\textbf{At the end of this section, we give the proof of Theorem \ref{T3}:}
Following the classical arguments of Cazenave and Lions \cite{CL}. Assume
that there exist an $\varepsilon _{0}>0$, a sequence of initial data $%
\left\{ u_{n}^{0}\right\} \subset X$ and a time sequence $\left\{
t_{n}\right\} \subset \mathbb{R}^{+}$ such that the unique solution $u_{n}$
of system (\ref{1.1}) with initial data $u_{n}^{0}=u_{n}(\cdot ,0)$
satisfies
\begin{equation*}
\text{dist}_{X}(u_{n}^{0},\mathcal{M}_{c}^{\rho })<\frac{1}{n}\text{ and dist%
}_{X}(u_{n}(\cdot ,t_{n}),\mathcal{M}_{c}^{\rho })\geq \varepsilon _{0}.
\end{equation*}%
Without loss of generality, we may assume that $\left\{ u_{n}^{0}\right\}
\subset S(c)$. Since $\text{dist}_{X}(u_{n}^{0},\mathcal{M}_{c}^{\rho
})\rightarrow 0$ as $n\rightarrow \infty $, the conservation laws of the
energy and mass imply that $u_{n}(\cdot ,t_{n})$ is a minimizing sequence
for $\gamma _{c}^{\rho }$ provided $u_{n}(\cdot ,t_{n})\subset \mathcal{B}%
_{\rho }$. Indeed, if $u_{n}(\cdot ,t_{n})\subset (X\backslash \mathcal{B}%
_{\rho })$, then by the continuity there exists $\bar{t}_{n}\in \lbrack
0,t_{n})$ such that $\left\{ u_{n}(\cdot ,\bar{t}_{n})\right\} \subset
\partial \mathcal{B}_{\rho }$. Hence, by Lemma \ref{L3.2} one has
\begin{equation*}
J(u_{n}(\cdot ,\bar{t}_{n}))\geq \inf_{u\in S(c)\cap \partial \mathcal{B}%
_{\rho }}J(u)>\inf_{u\in S(c)\cap \partial \mathcal{B}_{b\rho
}}J(u)=\inf_{u\in S(c)\cap \partial \mathcal{B}_{\rho }}J(u)=\gamma
_{c}^{\rho },
\end{equation*}%
which is a contradiction. Therefore, $\left\{ u_{n}(\cdot ,t_{n})\right\} $
is a minimizing sequence for $\gamma _{c}^{\rho }$. Then there exists $%
v_{0}\in \mathcal{M}_{c}^{\rho }$ such that $u_{n}(\cdot ,t_{n})\rightarrow
v_{0}$ in $X$, which contradicts with dist$_{X}(u_{n}(\cdot ,t_{n}),\mathcal{%
M}_{c}^{\rho })\geq \varepsilon _{0}.$ We complete the proof.

\section{The high-energy standing waves}

First of all, we prove that the energy functional $J$ on $S(c)$ possesses a
kind of mountain-pass geometrical structure. For each $u\in H^{1}(\mathbb{R}%
^{2})\backslash \{0\}$ and $t>0$, we set
\begin{equation*}
u_{t}(x):=tu(tx)\ \text{for all}\ x\in \mathbb{R}^{2},
\end{equation*}%
we have the following result.

\begin{lemma}
\label{L3.4} Assume that conditions $(f_{1})-(f_{2})$ and $(f_{4})$ hold.
Let $u\in S(c)$ be arbitrary but fixed. Then the following statements are
true:\newline
$(i)$ $A(u_{t})\rightarrow 0$ and $J(u_{t})\rightarrow +\infty $ as $%
t\rightarrow 0;$\newline
$(ii)$ $A(u_{t})\rightarrow +\infty $ and $J(u_{t})\rightarrow -\infty $ as $%
t\rightarrow +\infty .$
\end{lemma}

\begin{proof}
A direct calculation shows that
\begin{equation}
\int_{\mathbb{R}^{2}}|u_{t}|^{2}dx=\int_{\mathbb{R}^{2}}|u|^{2}dx=c,\
A(u_{t})=t^{2}A(u),\ V(u_{t})=V(u)-c^{2}\ln t,  \label{3.2.1}
\end{equation}%
and
\begin{equation}
\int_{\mathbb{R}^{2}}|u_{t}|^{r}dx=t^{r-2}\int_{\mathbb{R}^{2}}|u|^{r}dx\
\text{for}\ r>2.  \label{3.2.2}
\end{equation}%
Clearly,
\begin{equation}
A(u_{t})\rightarrow 0\ \text{and}\ \Vert u_{t}\Vert _{L^{r}}^{r}\rightarrow
0\ \text{as}\ t\rightarrow 0.  \label{3.2.3}
\end{equation}%
Then, there exist $t_{0}>0$ and $0<m<1$ such that
\begin{equation*}
A(u_{t})\leq m,\ \forall t\in (0,t_{0}].
\end{equation*}%
Similar to the argument in Lemma \ref{L3.1}, by conditions $(f_{1})-(f_{2}),$
for any $\xi >0$ and for fixed $q>2$, there exists a constant $%
K_{2}=K_{2}(\xi ,\alpha ,q,c)>0$ such that
\begin{equation*}
\left\vert \int_{\mathbb{R}^{2}}F(u_{t})dx\right\vert \leq \xi \int_{\mathbb{%
R}^{2}}|u_{t}|^{\tau +1}dx+K_{2}\left( \int_{\mathbb{R}^{2}}|u_{t}|^{q\eta
}dx\right) ^{1/\eta },\ \forall t\in (0,t_{1}],
\end{equation*}%
where $\eta =\frac{\bar{\eta}}{\bar{\eta}-1}$ with $\bar{\eta}>1$ closing to
$1$, and together with (\ref{3.2.2}), we have
\begin{equation}
\int_{\mathbb{R}^{2}}F(u_{t})dx\rightarrow 0\ \text{as}\ t\rightarrow 0.
\label{3.2.5}
\end{equation}%
Moreover, it follows from (\ref{3.2.1}) that
\begin{equation}
V(u_{t})=V(u)-c^{2}\ln t\rightarrow +\infty \ \text{as}\ t\rightarrow 0.
\label{3.2.6}
\end{equation}%
Hence, by (\ref{3.2.3})--(\ref{3.2.6}), one has
\begin{equation*}
J(u_{t})\rightarrow +\infty \ \text{as}\ t\rightarrow 0.
\end{equation*}

On the other hand, it is clear that $A(u_{t})\rightarrow +\infty $ as $%
t\rightarrow +\infty ,$ and it follows from conditions $(f_{4})$ that
\begin{eqnarray*}
J(u_{t}) &=&\frac{1}{2}A(u_{t})+\frac{1}{4}V(u_{t})-\int_{\mathbb{R}%
^{2}}F(u_{t})dx \\
&\leq &\frac{t^{2}}{2}A(u)+\frac{1}{4}V(u)-\frac{c^{2}\ln t}{4}-\theta
t^{p-2}\int_{\mathbb{R}^{2}}|u|^{p}dx \\
&\rightarrow &-\infty \ \text{as}\ t\rightarrow +\infty ,
\end{eqnarray*}%
since $p>4$. We complete the proof.
\end{proof}

By Lemma \ref{L3.4}, there exists $t_{1}>>1$ such that $%
w_{c}:=(u_{c})_{t_{1}}\in S(c)\backslash \mathcal{B}_{\rho }$ and $%
J(w_{c})<0,$ where $u_{c}$ is the ground state obtained in Theorem \ref{T2}
with $J(u_{c})>0$ for $0<c<\tilde{c}_{\ast }.$ Then, following the idea of
Jeanjean \cite{J}, the energy functional $J$ has the mountain-pass geometry
on $S(c)$. Define a set of paths
\begin{equation*}
\Gamma :=\left\{ h\in C([0,1],S(c))\text{ }|\text{\ }h(0)=u_{c},h(1)=w_{c}%
\right\}
\end{equation*}%
and a minimax value
\begin{equation*}
m(c):=\inf_{h\in \Gamma }\max_{\tau \in \lbrack 0,1]}J(h(\tau )),
\end{equation*}%
Clearly, $\Gamma \neq \emptyset $ and%
\begin{equation*}
\max_{\tau \in \lbrack 0,1]}J(h(\tau ))>\max \left\{
J(u_{c}),J(w_{c})\right\} >0\text{ for }0<c<\tilde{c}_{\ast }.
\end{equation*}

Next, we introduce an auxiliary functional $\tilde{J}:S(c)\times \mathbb{R}%
\rightarrow \mathbb{R}$ given by $(u,l)\rightarrow J(\psi (u,l)),$ where $%
\psi (u,l):=lu(lx)$. To be precise, we have
\begin{eqnarray*}
\tilde{J}(u,l) &=&J(\psi (u,l)) \\
&=&\frac{l^{2}}{2}\Vert \nabla u\Vert _{2}^{2}+\frac{1}{4}(V(u)-c^{2}\ln l)-%
\frac{1}{l^{2}}\int_{\mathbb{R}^{2}}F(lu)dx.
\end{eqnarray*}%
Define a set of paths
\begin{equation*}
\tilde{\Gamma}:=\left\{ \tilde{h}\in C([0,1],S(c)\times \mathbb{R})\text{ }|%
\text{\ }\tilde{h}(0)=(u_{c},1)\text{ and }\tilde{h}(1)=(w_{c},1)\right\}
\end{equation*}%
and a minimax value
\begin{equation*}
\tilde{m}(c):=\inf_{\tilde{h}\in \tilde{\Gamma}}\max_{0\leq t\leq 1}\tilde{J}%
(\tilde{h}(t)).
\end{equation*}%
We now claim that $\tilde{m}(c)=m(c)$. In fact, it follows immediately from
the definitions of $\tilde{m}(c)$ and $m(c)$ along with the fact that the
maps
\begin{equation*}
\chi :\Gamma \rightarrow \tilde{\Gamma}\text{ by }h\rightarrow \chi
(h):=(h,1)
\end{equation*}%
and
\begin{equation*}
\Upsilon :\tilde{\Gamma}\rightarrow \Gamma \text{ by }\tilde{h}\rightarrow
\Upsilon (\tilde{h}):=\psi \circ \tilde{h}
\end{equation*}%
satisfying
\begin{equation*}
\tilde{J}(\chi (h))=J(h)\text{ and }J(\Upsilon (\tilde{h}))=\tilde{J}(\tilde{%
h}).
\end{equation*}

Denote $\left\Vert r\right\Vert _{\mathbb{R}}=|r|$ for $r\in \mathbb{R}$, $%
H:=X\times \mathbb{R}$ endowed with the norm $\Vert \cdot \Vert
_{H}^{2}=\Vert \cdot \Vert _{X}^{2}+\Vert \cdot \Vert _{\mathbb{R}}^{2}$ and
$H^{-1}$ the dual space of $H$. By Jeanjean \cite{J}, we have the following
lemma.

\begin{lemma}
\label{L3.3} (\cite[Lemma 2.3]{J})Let $\varepsilon >0$. Assume that $\tilde{h%
}_{0}\in \tilde{\Gamma}$ satisfies $\max_{0\leq t\leq 1}\tilde{J}(\tilde{h}%
_{0}(t))\leq \tilde{m}(c)+\varepsilon .$ Then there exists a couple of $%
(u_{0},l_{0})\in S(c)\times \mathbb{R}$ such that\newline
$(i)$ $\tilde{J}(u_{0},l_{0})\in \lbrack \tilde{m}(c)-\varepsilon ,\tilde{m}%
(c)+\varepsilon ];$\newline
$(ii)$ $\min_{0\leq t\leq 1}\Vert (u_{0},l_{0})-\tilde{h}_{0}(t)\Vert
_{X}\leq \sqrt{\varepsilon };$\newline
$(iii)$ $\Vert (\tilde{J}|_{S(c)\times \mathbb{R}})^{\prime
}(u_{0},l_{0})\Vert _{H^{-1}}\leq 2\sqrt{\varepsilon },$ i.e. $|\left\langle
\tilde{J}^{\prime }(u_{0},l_{0}),z\right\rangle _{H^{-1}\times H}|\leq 2%
\sqrt{\varepsilon }\Vert z\Vert _{H}$ holds for all
\begin{equation*}
z\in \tilde{T}_{(u_{0},l_{0})}:=\left\{ (z_{1},z_{2})\in H\text{ }|\text{\ }%
\left\langle u_{0},z_{1}\right\rangle =0\right\} .
\end{equation*}
\end{lemma}

By vitue of Lemma \ref{L3.3}, we establish the following result.

\begin{lemma}
\label{L3.5}Assume that conditions $(f_{1})-(f_{2}),(f_{4})$ and $(f_{6})$
hold. Then there exists a sequence $\left\{ u_{n}\right\} \subset S(c)$ such
that
\begin{equation*}
J(u_{n})\rightarrow m(c),\quad (J|_{S(c)})^{\prime }(u_{n})\rightarrow
0\quad \text{and }Q(u_{n})\rightarrow 0\text{ as }n\rightarrow \infty .
\end{equation*}
\end{lemma}

\begin{proof}
The proof is similar to that of \cite[pp.1643-1645]{J}, we omit it here.
\end{proof}

\begin{lemma}
\label{L3.7} Assume that conditions $(f_{1})-(f_{2}),(f_{4})$ and $(f_{6})$
hold. Then there exist two constants $0<\hat{c}_{\ast }\leq \tilde{c}_{\ast
} $ and $\theta _{0}>0$ such that for $0<c<\hat{c}_{\ast }$ and $\theta
>\theta _{0},$%
\begin{equation*}
m(c)<\frac{(p-4)(1-c)+c^{2}}{4(p-2)}.
\end{equation*}
\end{lemma}

\begin{proof}
Set
\begin{equation*}
h(s)=\left( 1-s+st_{1}\right) u_{c}(\left( 1-s+st_{1}\right) x)\text{ for }%
s\in \lbrack 0,1].
\end{equation*}%
Clearly, $h(s)\in \Gamma $. By condition $(f_{4})$, we have
\begin{eqnarray}
m(c) &\leq &\max_{s\in \lbrack 0,1]}J(h(s))  \notag \\
&\leq &\max_{t\in \lbrack 1,t_{1}]}\left[ \frac{t^{2}}{2}A(u_{c})+\frac{1}{4}%
V(u_{c})-\frac{c^{2}}{4}\ln t-\theta t^{p-2}\Vert u_{c}\Vert _{L^{p}}^{p}%
\right]  \notag \\
&\leq &\max_{t>0}\left[ \frac{t^{2}}{2}A(u_{c})-\theta t^{p-2}\Vert
u_{c}\Vert _{L^{p}}^{p}\right] +\frac{1}{2}\Vert u_{c}\Vert _{\ast }^{2}c^{2}
\notag \\
&=&\frac{p-4}{2(p-2)}\left( \frac{\theta (p-2)\Vert u_{c}\Vert _{L^{p}}^{p}}{%
A(u_{c})}\right) ^{2/(4-p)}A(u_{c})+\frac{1}{2}\Vert u_{c}\Vert _{\ast
}^{2}c^{2}.  \label{3.2.7}
\end{eqnarray}%
Moreover, we note that there exists a constant $0<\hat{c}_{\ast }\leq \tilde{%
c}_{\ast }$ such that for $0<c<\hat{c}_{\ast },$%
\begin{equation}
1-c+\frac{1-2(p-2)\Vert u_{c}\Vert _{\ast }^{2}}{p-4}c^{2}>0.  \label{3.2.8}
\end{equation}%
Then it follows from (\ref{3.2.7}) and (\ref{3.2.8}) that
\begin{equation*}
m(c)<\frac{(p-4)(1-c)+c^{2}}{4(p-2)},
\end{equation*}%
for $0<c<\hat{c}_{\ast }$ and
\begin{equation*}
\theta >\theta _{0}:=\left( \frac{A(u_{c})}{(p-2)\Vert u_{c}\Vert
_{L^{p}}^{p}}\right) \left[ \frac{2A(u_{c})}{1-c+\frac{1-2(p-2)\Vert
u_{c}\Vert _{\ast }^{2}}{p-4}c^{2}}\right] ^{(p-4)/2}.
\end{equation*}%
We complete the proof.
\end{proof}

\begin{lemma}
\label{L3.8} Assume that conditions $(f_{1})-(f_{2}),(f_{4})$ and $(f_{6})$
hold. Let $\{u_{n}\}\subset S(c)$ be a (PS)-sequence for the energy
functional $J$ at the level $m(c)$ with $Q(u_{n})=o(1)$. Then there exists a
positive constant $c^{\ast }\leq \hat{c}_{\ast }$ such that for $0<c<c^{\ast
}$ and $\theta >\theta _{0},$%
\begin{equation*}
\limsup_{n\rightarrow \infty }A(u_{n})<1-c.
\end{equation*}
\end{lemma}

\begin{proof}
By Lemmas \ref{L2.5} and \ref{L3.7}, we have
\begin{eqnarray*}
\frac{(p-4)(1-c)+c^{2}}{4(p-2)}+o(1) &>&m(c)+o(1)=J(u_{n}) \\
&\geq &\frac{p-4}{2(p-2)}A(u_{n})-\frac{K}{4}c^{3/2}A(u_{n})^{1/2}+\frac{1}{%
4(p-2)}c^{2}+o(1),
\end{eqnarray*}%
which implies that there exists a positive constant $c^{\ast }\leq \hat{c}%
_{\ast }$ such that for $0<c<c^{\ast }$ and $\theta >\theta _{0},$%
\begin{equation*}
\limsup_{n\rightarrow \infty }A(u_{n})<1-c.
\end{equation*}%
We complete the proof.
\end{proof}

\begin{lemma}
\label{L3.9} Assume that conditions $(f_{1})-(f_{2}),(f_{4})$ and $(f_{6})$
hold. Let $\{u_{n}\}\subset S(c)$ be a (PS)-sequence for the energy
functional $J$ at level $m(c)$ with $Q(u_{n})=o(1)$. Then, up to a
subsequence, $u_{n}\rightarrow \bar{u}_{c}$ in $X$. In particular, $\bar{u}%
_{c}$ is a critical point of $J$ restricted to $S(c).$
\end{lemma}

\begin{proof}
Let $\{u_{n}\}\subset S(c)$ be a (PS)-sequence for $J$ at level $m(c)$ with $%
Q(u_{n})=o(1).$ Then it follows from Lemmas \ref{L2.6} and \ref{L3.8} that $%
\{u_{n}\}$ is bounded in $X.$ Passing to a subsequence if necessary, there
exists $\bar{u}_{c}\in X$ such that $u_{n}\rightharpoonup \bar{u}_{c}$
weakly in $X,$ $u_{n}\rightarrow \bar{u}_{c}$ in $L^{r}(\mathbb{R}^{2})$ for
all $r\in \lbrack 2,\infty )$ by Lemma \ref{L2.1}$(i)$ and $u_{n}\rightarrow
\bar{u}_{c}$ a.e. in $\mathbb{R}^{2}$. Clearly, $\bar{u}_{c}\neq 0.$ By the
Lagrange multipliers rule, there exists $\lambda _{n}\in \mathbb{R}$ such
that for every $\varphi \in X$,
\begin{equation}
\int_{\mathbb{R}^{2}}\nabla u_{n}\nabla \varphi dx+\lambda _{n}\int_{\mathbb{%
R}^{2}}u_{n}\varphi dx+\left[ V_{1}^{\prime }(u_{n})-V_{2}^{\prime }(u_{n})%
\right] \varphi -\int_{\mathbb{R}^{2}}f(u_{n})\varphi dx=o(1)\Vert \varphi
\Vert .  \label{3.2.4}
\end{equation}%
This shows that
\begin{equation}
\lambda _{n}c:=-A(u_{n})-V(u_{n})+\int_{\mathbb{R}^{2}}f(u_{n})u_{n}dx+o(1).
\label{3.2.9}
\end{equation}%
Similar to (\ref{e3-4}), it follows from (\ref{GN}) and Lemma \ref{L2.6}
that $\int_{\mathbb{R}^{2}}f(u_{n})u_{n}dx$\ is bounded. Moreover, we obtain
that $V_{1}(u_{n})$ is bounded by (\ref{V1}) and $V_{2}(u_{n})$ is bounded
by (\ref{V2}), respectively. Thus, from (\ref{3.2.9}) it follows that $%
\{\lambda _{n}\}\in \mathbb{R}$ is bounded, up to a sequence, we can assume
that $\lambda _{n}\rightarrow \bar{\lambda}\in \mathbb{R}$ as $n\rightarrow
\infty $.

Next, we prove that $u_{n}\rightarrow \bar{u}_{c}$ strongly in $X$, which
will thus imply that $\bar{u}_{c}$ is a critical point of $J$ restricted to $%
S(c)$. By (\ref{3.2.4}), we know that $\bar{u}_{c}$ is a weak solution to
Eq. (\ref{2.2}), which indicates that
\begin{equation}
Q(\bar{u}_{c})=A(\bar{u}_{c})-\frac{1}{4}\Vert \bar{u}_{c}\Vert
_{L^{2}}^{4}+\int_{\mathbb{R}^{2}}(2F(\bar{u}_{c})-f(\bar{u}_{c})\bar{u}%
_{c})dx=0  \label{3.2.10}
\end{equation}%
by Lemma \ref{L2.7}. Now, by $Q(u_{n})=o(1)$ and (\ref{3.2.10}), we have
\begin{eqnarray}
&&A(u_{n})-\frac{1}{4}\Vert u_{n}\Vert _{L^{2}}^{4}+\int_{\mathbb{R}%
^{2}}(2F(u_{n})-f(u_{n})u_{n})dx  \notag \\
&=&A(\bar{u}_{c})-\frac{1}{4}\Vert \bar{u}_{c}\Vert _{L^{2}}^{4}+\int_{%
\mathbb{R}^{2}}(2F(\bar{u}_{c})-f(\bar{u}_{c})\bar{u}_{c})dx+o(1).
\label{3.2.11}
\end{eqnarray}%
Moreover, by Lemma \ref{L2-10}, we have
\begin{equation}
F(u_{n})\rightarrow F(\bar{u}_{c})\ \text{and }f(u_{n})u_{n}\rightarrow f(%
\bar{u}_{c})\bar{u}_{c}\text{ in}\ L^{1}(\mathbb{R}^{2}).  \label{3.2.13}
\end{equation}%
Thus, it follows from (\ref{3.2.11}) and (\ref{3.2.13}) that $%
A(u_{n})\rightarrow A(\bar{u}_{c})$, where we have also used the fact of $%
u_{n}\rightarrow \bar{u}_{c}$ in $L^{2}(\mathbb{R}^{2}).$

Since $A(u_{n})\rightarrow A(\bar{u}_{c})$ and $u_{n}\rightarrow \bar{u}_{c}$
in $L^{r}(\mathbb{R}^{2})$ for all $r\in \lbrack 2,+\infty )$, by choosing $%
\varphi =u_{n}-\bar{u}_{c}$ in (\ref{3.2.4}), one has
\begin{equation*}
o(1)=o(1)+\frac{1}{4}\left[ V_{1}^{\prime }(u_{n})(u_{n}-\bar{u}%
_{c})-V_{2}^{\prime }(u_{n})(u_{n}-\bar{u}_{c})\right] -\int_{\mathbb{R}%
^{2}}f(u_{n})(u_{n}-\bar{u}_{c})dx.
\end{equation*}%
Moreover, we have
\begin{equation*}
\left\vert V_{2}^{\prime }(u_{n})(u_{n}-\bar{u}_{c})\right\vert \leq
C_{1}\Vert u_{n}\Vert _{L^{\frac{8}{3}}}^{3}\Vert u_{n}-\bar{u}_{c}\Vert
_{L^{\frac{8}{3}}}\rightarrow 0,
\end{equation*}%
and
\begin{equation*}
\left\vert \int_{\mathbb{R}^{2}}f(u_{n})(u_{n}-\bar{u}_{c})dx\right\vert
\leq \varepsilon \Vert u_{n}\Vert _{L^{2\tau }}^{\tau }\Vert u_{n}-\bar{u}%
_{c}\Vert _{L^{2}}+K_{\varepsilon }\Vert u_{n}\Vert _{L^{r(q-1)}}^{q-1}\Vert
u_{n}-\bar{u}_{c}\Vert _{L^{r}}\rightarrow 0,
\end{equation*}%
and
\begin{equation*}
\left\vert V_{1}^{\prime }(u_{n})(u_{n}-\bar{u}_{c})\right\vert
=B_{1}(u_{n}^{2},u_{n}(u_{n}-\bar{u}_{c}))=B_{1}(u_{n}^{2},(u_{n}-\bar{u}%
_{c})^{2})+B_{1}(u_{n}^{2},u(u_{n}-\bar{u}_{c}))
\end{equation*}%
with $B_{1}(u_{n}^{2},u(u_{n}-\bar{u}_{c}))\rightarrow 0$ as $n\rightarrow
\infty $ by Lemma \ref{L2.3}. Hence, we get
\begin{equation*}
o(1)=o(1)+B_{1}(u_{n}^{2},(u_{n}-\bar{u}_{c})^{2}),
\end{equation*}%
which implies that $B_{1}(u_{n}^{2},(u_{n}-\bar{u}_{c})^{2})\rightarrow 0$
as $n\rightarrow \infty $, together with Lemma \ref{L2.2}, leading to $\Vert
u_{n}-\bar{u}_{c}\Vert _{\ast }\rightarrow 0$ as $n\rightarrow \infty $.
Therefore, we deduce that $\Vert u_{n}-\bar{u}_{c}\Vert _{X}\rightarrow 0$
as $n\rightarrow \infty $. We complete the proof.
\end{proof}

\textbf{We are ready to prove Theorem \ref{T4}:} By Lemmas \ref{L3.5} and %
\ref{L3.7}, for $0<c<\hat{c}_{\ast }$ and $\theta >\theta _{0},$ there
exists a bounded Palais-Smale sequence $\{u_{n}\}\subset S(c)$ for $J$ at
level $m(c) $. Then, it follows from Lemma \ref{L3.9} that $u_{n}\rightarrow
\bar{u}_{c}$ in $X$ and $\bar{u}_{c}$ is a critical point for $J$ restrict
to $S(c)$, which shows that $\bar{u}_{c}$ is a mountain pass solution of
problem $(SP_{c})$ satisfing
\begin{equation*}
J(u_{c})<J(\bar{u}_{c})=m(c).
\end{equation*}%
We complete the proof.

\section{Acknowledgments}

J. Sun was supported by the National Natural Science Foundation of China
(Grant No. 11671236) and Shandong Provincial Natural Science Foundation
(Grant No. ZR2020JQ01).

\end{document}